\documentclass[reqno,12pt]{amsart} 
\usepackage{vmargin}
\setpapersize{A4}

\usepackage{amsmath} 

\usepackage{amsfonts}   

\usepackage{amssymb}      

\usepackage{eufrak}      

\usepackage{amscd}      
\usepackage{amsthm}      

\usepackage{tikz,amsmath}
\usetikzlibrary{arrows}

\usepackage{epsfig}      

\usepackage{amstext}      

\usepackage{graphicx}

\usepackage[all,line,dvips]{xy} 
\CompileMatrices 

\newcommand{\id}{\operatorname{id}}

\newcommand{\im}{\operatorname{im}}

\newcommand{\Span}{\operatorname{Span}}

\newcommand{\Int}{\operatorname{Int}}
\newcommand{\Per}{\operatorname{Per}}

 \newcommand{\supp}{\operatorname{supp}}

\newcommand{\Is}{\operatorname{Is}}

\newcommand{\val}{\operatorname{val}}

\newcommand{\Ro}{\operatorname{RO^+}}
\newcommand{\Orb}{\operatorname{Orb}}

\DeclareMathOperator{\orb}{Orb}

   \theoremstyle{plain}
   \newtheorem{thm}{Theorem}[section]
   \newtheorem{prop}[thm]{Proposition}
   \newtheorem{lemma}[thm]{Lemma}  
   \newtheorem{cor}[thm]{Corollary}
   \theoremstyle{definition}

   \newtheorem{example}[thm]{Example}
   \theoremstyle{remark}

\newtheorem{assumption}[thm]{Assumption}

\usepackage{mdframed,xcolor}

\definecolor{mybgcolor}{gray}{0.8}
\definecolor{myframecolor}{rgb}{.647,.129,.149}

\mdfdefinestyle{mystyle}{
  usetwoside=false,
  skipabove=0.6em plus 0.8em minus 0.2em,
  skipbelow=0.6em plus 0.8em minus 0.2em,
  innerleftmargin=.25em,
  innerrightmargin=0.25em,
  innertopmargin=0.25em,
  innerbottommargin=0.25em,
  leftmargin=-.75em,
  rightmargin=-0em,
  topline=false,
  rightline=false,
  bottomline=false,
  leftline=false,
  backgroundcolor=mybgcolor,
  splittopskip=0.75em,
  splitbottomskip=0.25em,
  innerleftmargin=0.5em,
  leftline=true,
  linecolor=myframecolor,
  linewidth=0.25em,
}

\newmdenv[style=mystyle]{important}

\usepackage{lipsum}

   \numberwithin{equation}{section}

        \date{\today}

\title[Circle maps and $C^*$-algebras]{Circle maps and $C^*$-algebras}
  
\author{Thomas Lundsgaard Schmidt}
\author{Klaus Thomsen}

\DeclareMathOperator\coker{coker}

\date{\today}

\email{matkt@imf.au.dk}
\address{Institut for Matematik, Aarhus University, Ny Munkegade, 8000 Aarhus C, Denmark}

\begin{document}

\maketitle





\begin{abstract} We consider a construction of $C^*$-algebras from
  continuous piecewise monotone maps on the circle which generalizes
  the crossed product construction for homeomorphisms and more
  generally the construction of Renault, Deaconu and
  Anantharaman-Delaroche for local homeomorphisms. Assuming that the
  map is surjective and not locally injective we give necessary and sufficient
  conditions for the simplicity of the $C^*$-algebra and show that it
  is then a Kirchberg algebra. We provide tools for the
  calculation of the K-theory groups and turn them into an
  algorithmic method for Markov maps.  
\end{abstract}

\section{Introduction}

There are by now a wealth of ways to associate a $C^*$-algebra to specific
classes of dynamical systems, both reversible and irreversible. There
are several different approaches to the construction, but the most
successful is arguably the one which uses a groupoid as an intermediate
step. Via the groupoid the study of how the $C^*$-algebra depends on
the dynamical system and which features it captures can be broken into smaller steps, relating the
dynamical system to the groupoid and the groupoid to the
$C^*$-algebra. Since the $C^*$-algebra arises as the convolution
algebra of the groupoid there are by now a large reservoir of results
which can be exploited for this purpose. As concrete examples of very
successful studies which have followed this general recipe we
mention only the work of Putnam and Spielberg, \cite{PS}, and the
work of Deaconu and Shultz, \cite{DS}.

Another important aspect of
the groupoid approach is the interpretation which the construction is
given in
the non-commutative geometry of Connes where the algebra is seen as a
substitute for the poorly behaved quotient of the unit space
by the equivalence relation coming from the action of the
groupoid. In this picture the classical crossed
product algebra coming from the action of a group by homeomorphisms is the non-commutative
space representing the space of orbits under the action, and the
groupoid which serves as the stepping stone from the dynamical system
to the
$C^*$-algebra is the transformation groupoid.

As a purely algebraic
object the transformation
groupoid of a homeomorphism can easily be described such that the
invertibility of the dynamics is insignificant. It is therefore tempting and natural
to try to base a
generalization of the crossed product by a homeomorphism on the transformation groupoid. The problem is to
equip this groupoid with a sufficiently
nice topology which will allow the
construction of the convolution algebra. It was shown by Renault,
Deaconu and Anantharaman-Delaroche, in increasing generality,
\cite{Re}, \cite{De}, \cite{An}, that for a local homeomorphism there
is a canonical topology on the transformation groupoid which turns it
into an \'etale locally compact Hausdorff groupoid, which is the ideal
setting for
the formation of the convolution algebra. In \cite{Th2} the second author
introduced a way to describe the transformation groupoid of a
homeomorphism, and more generally a local homeomorphism, in such a way
that it leads to an \'etale groupoid and hence a $C^*$-algebra for
certain dynamical systems that are not local homeomorphisms. In
\cite{Th2} and \cite{Th3} this generalization of the crossed product
for homeomorphisms was used and investigated for holomorphic maps of
Riemann surfaces. Such maps are open, but fail to be injective
in any open set containing a critical point. The purpose with the present paper is to
show how the method from \cite{Th2} can be applied and what can be
said about the resulting algebras, in a case where the maps are
neither locally injective nor open. The class of maps we consider is the class of
continuous piecewise monotone maps of the circle, and there is at least two natural ways in which the viewpoint from \cite{Th2} can
be applied. We make a thorough study of one of them.

Specifically, from a continuous piecewise
monotone map $\phi : \mathbb T \to \mathbb T$ we construct an \'etale second countable
locally compact Hausdorff groupoid $\Gamma^+_{\phi}$ which is the usual transformation
groupoid when $\phi$ is an orientation preserving homeomorphism and is equal to the groupoid
of Renault, Deaconu and Anantharaman-Delaroche, \cite{An}, \cite{De}, \cite{Re}, when the map is an
orientation preserving local homeomorphism. The $C^*$-algebra we study is then the reduced
$C^*$-algebra $C^*_r\left(\Gamma^+_{\phi}\right)$ of the groupoid
$\Gamma^+_{\phi}$, cf. \cite{Re}.

We describe now our results. We assume that $\phi$ is surjective and not locally
injective. Without surjectivity the necessary and sufficient
condition for simplicity of $C^*_r\left(\Gamma^+_{\phi}\right)$ will
be more messy than the one we obtain for surjective maps, and the lack
of generality seems insignificant at this point. Assuming that $\phi$
is not locally injective means that we exclude the surjective local homeomorphisms. This is done with good conscience because 
$C^*_r\left(\Gamma^+_{\phi}\right)$ is equal to the usual
$C^*$-algebra of a local homeomorphism when
$\phi$ is a local homeomorphism of positive degree, and
equal to that of $\phi^2$ when the degree is negative. In addition, the simple
$C^*$-algebras that arise from the transformation groupoid of a
surjective local
homeomorphism of the circle are known, cf. \cite{AT}, and there is
therefore nothing lost by ignoring them here. We then establish the following facts. 
\begin{enumerate}
\item[1)] When $\phi$ is transitive, $C^*_r\left(\Gamma^+_{\phi}\right)$ is purely infinite (i.e. every
non-zero hereditary $C^*$-subalgebra contains an infinite
projection). (Proposition \ref{17}.)
\item[2)] $C^*_r\left(\Gamma^+_{\phi}\right)$ is simple if and only if $\phi$
is exact (or, equivalently, totally transitive) and has no
non-critical fixed point
$x$ such that $\phi^{-1}(x) \backslash \{x\}$ only contains critical points. (Theorem \ref{S}.)
\item[3)] $C^*_r\left(\Gamma^+_{\phi}\right)$ is unital,
separable, nuclear and satisfies the universal
coefficient theorem (UCT) of Rosenberg and Schochet. (Corollary \ref{nucUCT2}.)
\end{enumerate}
It follows from 1)-3) that when $\phi$ is exact without an exceptional fixed point the $C^*$-algebra $C^*_r\left(\Gamma^+_{\phi}\right)$ is classified by its K-theory, thanks to the Kirchberg-Phillips classification result, \cite{Ph}. To determine the algebra it suffices therefore to calculate its K-theory; a task which is far from trivial. We obtain here a six-terms exact sequence which can be be applied for the purpose (Theorem \ref{cunPi}.), and we turn it into an effective
algorithm for the calculation when 
$\phi$ is Markov in the sense that
it takes critical points to critical points. (Section \ref{MarMaps}.) 

As a couple of key references for our work we mention the work of Shultz \cite{S} and
the work of Katsura \cite{Ka}. The work of Shultz is used to show
that a continuous piecewise monotone and transitive map on the circle
is conjugate to a piecewise linear map with slopes that are constant
in absolute value; an important step towards the proof of the pure infiniteness of $C^*_r\left(\Gamma^+_{\phi}\right)$. His work is also
used to show that total transitivity and exactness are equivalent in
our setting; a fact which is important on the way to establish
nuclearity and the UCT . For this purpose we use also the work of
Katsura in much the same way it was used in \cite{Th3}, namely
to show that $C^*_r\left(\Gamma^+_{\phi}\right)$ can be realized
as a Cuntz-Pimsner algebra when it is simple. This is an important
point. While the groupoid picture is the best approach for the study
of many of the connections between properties of
$C^*_r\left(\Gamma^+_{\phi}\right)$ and the dynamical properties of $\phi$, it is the realization
of the algebra as a Cuntz-Pimsner algebra which provides the decisive
tools for the K-theory calculations.

The construction of $\Gamma^+_{\phi}$ depends on the choice of a
pseudo-group on $\mathbb T$; see Section \ref{SEC1}. This freedom is
present already for homeomorphisms and local homeomorphisms, and this
is why the $C^*$-algebra $C^*_r\left(\Gamma^+_{\phi}\right)$ only generalizes the
crossed product arising from a homeomorphism when it is orientation
preserving. While the choice of pseudo-group seems canonical for
homeomorphisms and local homeomorphisms, this is much less so in our case where at
least two choices seem equally natural; either
one can work with the pseudo-group of all locally defined local
homeomorphisms on $\mathbb T$, or one can restrict to those that
preserve the orientation of the circle. In the construction of
$\Gamma^+_{\phi}$ we choose the latter and we postpone the study of the
algebra which results when the largest pseudo-group is chosen.


\section{Transformation groupoid $C^*$-algebras}\label{SEC1}

Let $G$ be an \'etale second countable locally compact Hausdorff groupoid with unit
space $G^{(0)}$. Let $r : G \to G^{(0)}$ and $s : G \to G^{(0)}$ be
the range and source maps, respectively. For $x \in G^{(0)}$ put $G^x = r^{-1}(x), \ G_x = s^{-1}(x) \ \text{and} \ \Is_x =
s^{-1}(x)\cap r^{-1}(x)$. Note that $\Is_x$ is a group, the \emph{isotropy group} at $x$. The space $C_c(G)$ of continuous compactly supported
functions is a $*$-algebra when the product is defined by
$$
(f_1 f_2)(g) = \sum_{h \in G^{r(g)}} f_1(h)f_2(h^{-1}g)
$$
and the involution by $f^*(g) = \overline{f\left(g^{-1}\right)}$. Let $x\in G^{(0)}$. There is a
representation $\pi_x$ of $C_c(G)$ on the Hilbert space $l^2(G_x)$ of
square-summable functions on $G_x$ given by 
\begin{equation}\label{pix}
\pi_x(f)\psi(g) = \sum_{h \in G^{r(g)}} f(h)\psi(h^{-1}g) .
\end{equation}
The \emph{reduced groupoid
  $C^*$-algebra} $C^*_r(G)$, \cite{Re}, is the completion of $C_c(G)$ with respect to the norm
$$
\left\|f\right\| = \sup_{x \in G^{(0)}} \left\|\pi_x(f)\right\| .
$$

In \cite{Th2} the second named author introduced an amended transformation
groupoid for a self-map on a locally compact space which in certain
cases allows us to topologize the transformation groupoid, or a
groupoid closely related to the transformation groupoid, in such a way
that the result is a well behaved \'etale groupoid. It is this
construction we shall consider in this paper for piecewise monotone
maps of the circle. We begin therefore by reviewing the construction.

Let $X$ be a locally compact Hausdorff space and $\psi : X \to X$ a map. Let $\mathcal
 P$ be a pseudo-group on $X$. More specifically, $\mathcal P$ is a
collection of homeomorphisms $\eta : U \to V$ between open
subsets of $X$ such that
\begin{enumerate}
\item[i)] for every open subset $U$ of $X$ the identity map $\id : U
  \to U$ is in $\mathcal P$,
\item[ii)] when $\eta : U \to V$ is in $\mathcal P$ then so is
  $\eta^{-1} : V \to U$, and
\item[iii)] when $\eta : U \to V$ and $\eta_1 : U_1 \to V_1$ are
  elements in $\mathcal P$ then so is $\eta_1 \circ \eta : U \cap
  \eta^{-1}(V\cap U_1) \to \eta_1(V \cap U_1)$.  
\end{enumerate}  
For each $k \in \mathbb Z$ we denote by
$\mathcal T_k(\psi)$ the elements $\eta : U \to V$ of $\mathcal P$ with the
property that there are natural numbers $n,m$ such that $k = n-m$ and
\begin{equation}\label{crux0} 
\psi^n(z) = \psi^m(\eta(z))  \ \ \forall z \in U.
\end{equation}
The elements of $\mathcal T =\bigcup_{k \in \mathbb Z} \mathcal T_k(\psi)$ will
be called \emph{local transfers} for $\psi$. 
We denote by $[\eta]_x$ the germ at a point $x \in X$ of an element $\eta \in \mathcal
T_k(\psi)$. Set
$$
\mathcal G_{\psi} = \left\{ (x,k,\eta,y) \in X \times \mathbb Z \times
  \mathcal P  \times X : \ \eta \in \mathcal T_k(\psi) , \ \eta(x)
  = y \right\} .
$$
We define an equivalence relation $\sim$ in $\mathcal G_{\psi}$ such
that $(x,k,\eta,y) \sim (x',k',\eta',y')$ when
\begin{enumerate}
\item[i)] $ x = x', \ y = y', \ k = k'$ and
\item[ii)] $[\eta]_x = [\eta']_x$.
\end{enumerate}
Let
$\left[x,k,\eta,y\right]$ denote the equivalence class represented
by $(x,k,\eta,y)  \in \mathcal G_{\psi}$. The quotient space 
$G_{\psi}(\mathcal P) = \mathcal G_{\psi}/{\sim}$ is a groupoid such that two elements
$\left[x,k,\eta,y\right]$ and $\left[x',k',\eta',y'\right]$ are
composable when $y= x'$ and their product is
$$
\left[x,k,\eta,y\right]\left[y,k',\eta',y'\right] =
\left[x,k+k',\eta'\circ \eta ,y'\right] .
$$ 
The inversion in $G_{\psi}(\mathcal P)$ is defined such that
$\left[x,k,\eta,y\right]^{-1} = \left[ y,-k,\eta^{-1},x \right]$. The unit space of $G_{\psi}$ can be identified with $X$ via the map $x \mapsto [x,0,\id,x]$, where $\id$ is the
identity map on $X$. When $\eta \in \mathcal T_k(\psi)$ we set
\begin{equation}\label{baseset}
U(\eta) = \left\{ \left[z,k,\eta,\eta(z)\right]  : \ z \in U\right\}
\end{equation}
where $U$ is the domain of $\eta$.
It is straightforward to verify that by varying $k$, $\eta$ and $U$ the sets (\ref{baseset}) constitute a base
for a topology on $G_{\psi}(\mathcal P)$. In general this topology is not
Hausdorff and to amend this we now make the following additional assumption.
\begin{assumption}\label{crux2}
When $x \in X$ and $\eta(x) = \xi(x)$ for some $\eta,
\xi \in \mathcal T_k(\psi)$, then the implication
\begin{equation*}\label{crux1}
\text{$x$ is not isolated in} \ \left\{ y \in X : \ \eta(y) = \xi(y)\right\} \ \
\Rightarrow \ \ [\eta]_x =[\xi]_x
\end{equation*}
holds.
\end{assumption} 
Then $G_{\psi}(\mathcal P)$ is Hausdorff: Let
  $\left[x,k,\eta,y\right]$ and $\left[x',k',\eta',y'\right]$ be
  different elements of $G_{\psi}(\mathcal P)$. There are then open neighbourhood's $W$ of $x$ and $W'$ of $x'$ such that 
$U(\eta|_W) =\left\{
    \left[z,k,\eta,\eta(z)\right] : \ z \in W\right\}$ and $U'(\eta'|_{W'}) = \left\{
    \left[z,k',\eta',\eta'(z)\right] : \ z \in W'\right\}$ are
  disjoint. This is trivial when $(x,k,y) \neq (x',k',y')$ while it is
  a straightforward consequence of Assumption \ref{crux2} when
  $(x,k,y) = (x',k',y')$. 

Since the range and source maps are homeomorphisms from $U(\eta)$ onto
$U$ and $\eta(U)$, respectively, it follows that $G_{\psi}(\mathcal P)$ is a
locally compact Hausdorff space because $X$ is. It is also straightforward to show
that the groupoid operations are continuous so that we can
conclude the following. 

\begin{thm}\label{etale!} Let Assumption \ref{crux2} be satisfied. Then
  $G_{\psi}(\mathcal P)$ is an \'etale locally compact Hausdorff groupoid.  
\end{thm}

Compared to the notation used in \cite{Th2} we have here emphasized
the pseudo-group since there is more than one natural choice when $\psi$ is a piecewise monotone map on the circle.

\section{The oriented transformation groupoid of a piecewise monotone
  map on the circle}\label{orint}

Let $\mathbb T$ be the unit circle in the complex plane. We consider
$\mathbb T$ as an oriented space with the canonical counter-clockwise
orientation. Consider a continuous map $\phi : \mathbb T \to \mathbb T$. There is then a
unique continuous map $f : [0,1] \to \mathbb R$ such that $f(0) \in
[0,1[$ and
$$
\phi\left(e^{2 \pi i t} \right) = e^{2 \pi i f(t)}
$$
for all $t \in [0,1]$. We will refer to $f$ as \emph{the lift of
  $\phi$.} Note that $f(1) - f(0)$ is an integer, \emph{the degree of
  $\phi$}. We say that $\phi$ is \emph{piecewise monotone} when there
are points $0 = c_0 < c_1 < \dots < c_N = 1$ such that $f$ is either
strictly increasing or strictly decreasing on the intervals
$]c_{i-1},c_i[, \ i = 1,2, \dots, N$. When $\phi : \mathbb T \to
\mathbb T$ is piecewise monotone and $t \in \mathbb T$ we define \emph{the $\phi$-valency} $\val(\phi,t)$ of
$t$ to be the element of 
$$
\mathcal V = \left\{ (+,+), (-,-),
  (+,-), (-,+)\right\}
$$
such that $\val(\phi,t) = (+,+)$ when $\phi$ is strictly increasing in all
sufficiently small neighborhoods of $t$; $\val(\phi,t)
= (-,-)$ when $\phi$ is strictly decreasing in all sufficiently small
open neighborhoods of $t$; $\val(\phi,t) = (+,-)$ when $\phi$
is strictly increasing in all sufficiently small intervals to the left
of $t$ and strictly decreasing in all sufficiently small intervals to
the right of $t$; and finally $\val(\phi,t) = (-,+)$ when $\phi$
is strictly decreasing in all sufficiently small intervals to the left
of $t$ and strictly increasing in all sufficiently small intervals to
the right of $t$.

The set $\mathcal V$ is a monoid with the following
composition table. 
\begin{table}[ht]
\centering
\begin{tabular}{c|c|c|c|c}
$x \bullet y$ & $y = (+,+)$ & $y=(+,-)$ & $y=(-,+)$ & $y=(-,-)$ \\ \hline
$x= (+,+)$ & $(+,+)$ & $(+,-)$ & $(-,+)$ & $(-,-)$ \\ \hline
$x = (+,-)$ & $(+,-)$ & $(+,-)$ & $(+,-)$ & $(+,-)$ \\ \hline
$x= (-,+)$ & $(-,+)$ & $(-,+)$ & $(-,+)$ & $(-,+)$ \\ \hline
$x = (- ,-)$ & $(-,-)$ & $(-,+)$ & $(+,-)$ & $(+,+)$ \\ \hline
\end{tabular}
\bigskip
\caption{The composition table for $\bullet$}
\end{table}

This monoid structure will be important here because of the following observation:

\begin{lemma}\label{compx} Let $\phi, \varphi : \mathbb T \to \mathbb
  T$ be piecewise monotone, and let
  $x \in \mathbb T$. Then
$$
\val (\phi \circ \varphi, x) = \val (\phi, \varphi(x)) \bullet
\val(\varphi, x) .
$$
\end{lemma}

Let $\mathcal P^+$ be the pseudo-group of all locally defined homeomorphisms of
$\mathbb T$ that are orientation preserving. The elements of $\mathcal
P^+$ consist of open subsets $U,V$ in $\mathbb T$ and a homeomorphism
$\eta : U \to V$ such that $s,t \in U, \ s < t \Rightarrow \eta(s) <
\eta(t)$.

\begin{lemma}\label{mod2} Let $\phi, \varphi : \mathbb T \to \mathbb T$ be continuous
  and piecewise
  monotone maps and $x,y \in \mathbb T$ points such that $\phi(x) =
  \varphi(y)$. It follows that there is a $\eta \in
  \mathcal P^+$ such that 
\begin{enumerate}
\item[a)] $\eta(x)
  = y$, and 
\item[b)] $\phi(t) = \varphi(\eta(t))$ for all $t$ in a neighborhood
  of $x$ 
\end{enumerate}
if and only if $\val(\phi,x) = \val(\varphi,y)$. 

When this is the case, the germ $[\eta]_x$ of $\eta$ at $x$ is unique.
\end{lemma}  
\begin{proof} Straightforward.
\end{proof}

When $x,y \in \mathbb T$ and $k \in \mathbb Z$ we write $x \overset{k}{\sim} y$ when $k =
n-m$ for some $n,m
\in \mathbb N$ such that $\phi^n(x) = \phi^m(y)$ and $\val\left(\phi^n,x\right)
= \val\left(\phi^m, y\right)$.
Set
$$
\Gamma^+_{\phi} = \left\{ (x,k,y) \in \mathbb T \times \mathbb Z \times
  \mathbb T : \  x\overset{k}{\sim} y \    \right\} .
$$ 
Then $\Gamma^+_{\phi}$ is a groupoid where the composable pairs are
$$
{\Gamma^+_{\phi}}^{(2)} = \left\{ ((x,k,y),(x',k',y')) \in
{\Gamma^+_{\phi}}^2 : \ y = x' \right\}
$$
and the product is
$$
(x,k,y)(y,k',y') = (x,k+k',y') .
$$
The inversion is given by $(x,k,y)^{-1} = (y,-k,x)$. This groupoid is
identical with the groupoid denoted by $G_{\phi}(\mathcal P^+)$ in Section \ref{SEC1}. To see
this we denote, as in \cite{Th2}, by $\mathcal T_k(\phi)$ the
set of elements $\eta \in \mathcal P^+$ with the property that for some
$n,m \in \mathbb N$ such that $n-m =k$, the equality
$$
\phi^n(z) =  \phi^m(\eta (z))
$$
holds for all $z$ in the domain of $\eta$.  It follows then from the last statement in Lemma \ref{mod2}
that the implication $\eta, \eta' \in \mathcal T_k(\phi) , \ \eta(x) = \eta'(x) \  \Rightarrow \ [\eta ]_x = [\eta']_x$  
holds. We deduce therefore that the map
$$
G_{\phi}\left(\mathcal P^+\right) \ni \left[x,k,\eta,y\right] \mapsto (x,k,y) \in \Gamma^+_{\phi}
$$
is a bijection. It follows from Theorem \ref{etale!} that $\Gamma^+_{\phi}$ is an \'etale locally
compact Hausdorff groupoid in the topology for which a base is given by sets
of the form
\begin{equation}\label{basesets}
\Omega\left( \eta,U\right) = \left\{ (z,k,\eta(z)) :  \ z \in U \right\},
\end{equation}
where $\eta \in \mathcal T_k(\phi)$ and $U$ is an open subset of
$\eta$'s domain. However, there is an alternative description of this
topology which we now present. Among others it has the virtue that it
is obviously second countable.

 When $k \in \mathbb Z, n \in \mathbb N$ and $n+ k \geq 1, \ n  \geq 1$, set
$$
\Gamma^+_{\phi}(k,n) = \left\{ (x,l,y) \in \Gamma^+_{\phi} : \ l = k,
  \ \phi^{k+n}(x) = \phi^n(y) , \ \val\left(\phi^{k+n},x\right) = \val
  \left( \phi^n,y\right) \right\}
$$  
and
$$
\Gamma_{\phi}(k,n) = \left\{ (x,l,y) \in \mathbb T \times \mathbb Z
  \times \mathbb T: \ l = k,
  \ \phi^{k+n}(x) = \phi^n(y)\right\} .
$$  
Note that $\Gamma_{\phi}(k,n)$ is closed in $\mathbb T \times \mathbb
Z \times \mathbb T$.

\begin{lemma}\label{closed} $\Gamma^+_{\phi}(k,n)$ is the intersection
  of a closed and an open subset of $\mathbb T \times \mathbb Z \times
  \mathbb T$.
\end{lemma}
\begin{proof} Write $\mathbb T$ as the union of non-degenerate closed
  intervals $\mathbb T = \bigcup_i I^+_i \cup \bigcup_i I^-_i$
  such that $\phi^{n+k}$ is increasing on each $I^+_i$ for all $i$ and
  decreasing on $I^-_i$ for all $i$, and such that none of the intervals
  overlap in more than one point. Similarly, write $\mathbb T$ as the union of non-degenerate closed
  intervals $\mathbb T = \bigcup_j J^+_j \cup \bigcup_j J^-_j$
  such $\phi^{n}$ is increasing on each $J^+_j$ for all $j$ and
  decreasing on $J^-_j$ for all $j$, and such that none of the intervals
  overlap in more than one point. Then 
$$
\Gamma^+_{\phi}(k,n) = \left(A \cap \Gamma_{\phi}(k,n)\right)
\backslash B,
$$
where 
$$
A = \bigcup_{i,j} \left(I^+_i \times \{k\} \times J^+_j\right)  \cup
\bigcup_{i,j} \left(I^-_i \times \{k\} \times J^-_j\right)
$$ 
and $B$ is the finite set consisting of elements $(x,k,y) \in
\Gamma_{\phi}(n,k)$ such that $\val\left(\phi^{n+k},x\right) \in
\left\{(+,-), (-,+)\right\}$ or $ \val\left(\phi^n,y\right) \in
  \left\{(+,-), (-,+)\right\}$ while $\val\left(\phi^{n+k},x\right) \neq
  \val\left(\phi^n,y\right)$. 

\end{proof}

It follows from Lemma \ref{closed} that $\Gamma^+_{\phi}(k,n)$ is a locally compact Hausdorff space in
the relative topology inherited from $\mathbb T \times \mathbb Z
\times \mathbb T$. 

\begin{lemma}\label{top} A subset $W$ of $\Gamma_{\phi}^+(k,n)$ is
  open in the relative topology inherited from $\mathbb T \times
  \mathbb Z \times \mathbb T$ if and only for all $(x,k,y) \in W$
  there is an element $\eta \in \mathcal T_k(\phi)$ and an open
  subset $U$ of the domain of $\eta$ such that
$(x,k,y) \in \Omega(\eta,U) \subseteq W$.
\end{lemma}
\begin{proof} Assume first that $W \subseteq \Gamma^+_{\phi}(k,n)  $ is open in the relative topology
  inherited from $\mathbb T \times \mathbb Z
\times \mathbb T$ and consider a point $(x,k,y) \in W$. It follows from Lemma \ref{mod2} that there is
an element $\eta \in \mathcal T_{k}(\phi)$ such that $\eta(x) = y$ and
$\phi^{k+n}(z) = \phi^n(\eta(z))$ for all $z$ in a neighborhood $U$ of
$x$. By continuity of $\eta$ we can shrink $U$ to arrange that 
$(x,k,y) \in \Omega(\eta,U) \subseteq W$. This establishes one implication. To prove the other, let $\eta \in
\mathcal T_k(\phi)$ and let $U$ be an open subset of the domain of $\eta$. We must show that $\Omega(\eta, U)\cap
\Gamma^+_{\phi}(k,n)$ is open in the relative topology of $\Gamma^+_{\phi}(k,n)$ inherited from $\mathbb T \times \mathbb Z
\times \mathbb T$. Let $(x,k,y) \in \Omega(\eta, U)\cap
\Gamma^+_{\phi}(k,n)$. It follows from Lemma \ref{mod2} that
$\phi^{n+k}(z) = \phi^n\left(\eta(z)\right)$ for all $z$ sufficiently
close to $x$. If $(z,z')$ is sufficiently close to $(x,y)$ in $\mathbb T
  \times \mathbb T$ and $(z,k,z') \in \Gamma^+_{\phi}(k,n)$, the conditions $\val\left(\phi^{n+k},z\right) =
  \val\left(\phi^{n},z'\right)$ and $\phi^{n+k}(z) = \phi^{n}(z')$ imply that
$z\leq x \ \Leftrightarrow \ z' \leq y$. Combined with the fact that $\phi^{n+k}(z) =
\phi^{n}\left(\eta(z)\right) = \phi^{n}(z')$ we conclude from this
that $z'
=\eta(z)$ when $(z,z')$ is sufficiently close to $(x,y)$ and $(z,k,z')
\in \Gamma^+_{\phi}(k,n)$. That is, there is an open neighborhood $V$
of $(x,k,y)$ in $\mathbb T \times \mathbb Z \times \mathbb T$ such
that $V \cap \Gamma^+_{\phi}(k,n) \subseteq \Omega(\eta, U)\cap
\Gamma^+_{\phi}(k,n)$.

\end{proof}

In combination with Lemma \ref{mod2} it follows from Lemma \ref{top}
that $\Gamma^+_{\phi}(k,n)$ is an open subset of
  $\Gamma^+_{\phi}(k,n+1)$. Therefore
$$
\Gamma^+_{\phi}(k) = \bigcup_{n \geq -k+1} \Gamma^+_{\phi}(k,n)
$$
is locally compact and Hausdorff in the inductive limit topology,
and the disjoint union
$$
\Gamma^+_{\phi} = \bigsqcup_{k \in \mathbb Z} \Gamma^+_{\phi}(k) 
$$
is a locally compact Hausdorff in the topology where each $\Gamma^+_{\phi}(k)$ is
closed and open and has the topology just defined. By Lemma \ref{top}
this topology is identical with the one we obtain from Theorem
\ref{etale!}. Note that the topology is second countable since the
topology of $\mathbb T \times
\mathbb Z \times \mathbb T$ is. This proves the following

\begin{lemma}\label{etale} $\Gamma^+_{\phi}$ is a second countable locally
  compact Hausdorff \'etale groupoid.
\end{lemma}

We assume now and throughout the paper that $\phi : \mathbb T \to
\mathbb T$ is continuous and piecewise monotone. The
objective is to investigate the structure of
$C^*_r\left(\Gamma^+_{\phi}\right)$; in particular, when it is simple.

Let us first consider the case where $\phi$ is a local homeomorphism
so that the groupoid $\Gamma_{\phi}$ of Renault, Deaconu and
Anantharaman-Delaroche is defined, cf. \cite{Re}, \cite{De} and
\cite{An}. 

\begin{lemma}\label{ADR} Assume that $\phi$ is a local
  homeomorphism. If the degree of $\phi$ is positive, there is an
  isomorphism $\Gamma^+_{\phi} \simeq \Gamma_{\phi}$ of topological
  groupoids. If the degree of $\phi$ is negative there is an
  isomorphism $\Gamma^+_{\phi} \simeq \Gamma_{\phi^2}$ of topological
  groupoids. 
\end{lemma}
\begin{proof} This follows straightforwardly from the observation that
  $\val\left(\phi,x\right) = (+,+)$ for all $x \in \mathbb T$ when the
  degree is positive and $\val\left(\phi,x\right) = (-,-)$ for all $x
  \in \mathbb T$ when the degree is negative.
\end{proof} 

It follows from Lemma \ref{ADR} that for a local homeomorphism $\phi$
we have the equality
$C^*_r\left(\Gamma^+_{\phi}\right) = C^*_r\left(\Gamma_{\phi}\right)$
when the degree of $\phi$ is positive, and 
$C^*_r\left(\Gamma^+_{\phi}\right) =
C^*_r\left(\Gamma_{\phi^2}\right)$ when it is negative. The simple
$C^*$-algebras of the form $C^*_r\left(\Gamma_{\psi}\right)$ for a surjective,
local homeomorphism $\psi$ of the circle were all described in \cite{AT} and
we will therefore here assume that $\phi$ is not locally
injective. Thus, for the remaining part of the paper $\phi$ will be
continuous, piecewise monotone, surjective and not locally injective.




\section{Transitivity implies piecewise linearity and pure
  infiniteness} 

Let $s > 0$. A continuous function $g : [0,1] \to \mathbb R$ is \emph{uniformly
  piecewise linear with slope $s$} when there are points $0 = a_0 < a_1 <
a_2 < \dots < a_N = 1$ such that $g$ is linear with slope $\pm s$ on
each interval $\left[a_{i-1},a_i\right], i =1,2, \dots, N$. We say
that $\phi$ is uniformly piecewise linear with slope $s$ when its lift
$f : [0,1] \to \mathbb R$ is.

Recall that a continuous map $h : X \to X$ on a compact metric space $X$ is \emph{transitive} when for each pair $U,V$ of non-empty open sets in $X$ there is an $n \in \mathbb N$ such that $h^n(U) \cap V \neq \emptyset$. When $h$ is surjective, as in the case we consider, transitivity is equivalent to the existence of a point with dense forward orbit.

\begin{thm}\label{11} Assume that $\phi$ is transitive. It follows that
  there is an orientation-preserving homeomorphism $h : \mathbb T \to
  \mathbb T$ such that $h \circ \phi \circ h^{-1}$ is uniformly piecewise
  linear with slope $s > 1$.
\end{thm}
\begin{proof} We will show how the theorem follows from the work of
  Shultz in \cite{S} on discontinuous piecewise monotone maps
  of the interval.

After conjugation by a rotation of the circle we can assume that
$\phi(1) \neq 1$ and $1 \notin \phi(\mathcal C_1)$. (Indeed, since $\phi$
is piecewise monotone and transitive there are $\lambda$'s in $\mathbb T$ arbitrary close to
$1$ such that $\phi(\lambda) \neq \lambda$. Choose one of them such
that $\lambda \notin \phi(\mathcal C_1)$. Then $\phi_1(t) = \lambda^{-1}
\phi(\lambda t)$ is conjugate to $\phi$, does not fix $1$ and all its
critical values are different from $1$.) 
Let $\mu : \mathbb T \to [0,1[$ be the inverse map of $[0,1[ \ni t \mapsto e^{2
  \pi i t}$. Then
$$
\tau(t) = \mu \circ \phi\left(e^{2 \pi i t}\right)
$$
is piecewise monotone in the sense of Shultz \cite{S}. Since $\phi$ is surjective and $1 \notin \phi\left(\mathcal C_1 \cup \{1\}\right)$, it follows that
$\tau$ is discontinuous at a point in $]0,1[$ and $\tau\left([0,1]\right) = [0,1[$. We claim that
$\tau$ is transitive in the sense of Definition 2.6 in
\cite{S}; that is, we claim that for every open non-empty subset $U
\subseteq [0,1]$ there is an $n \in \mathbb N$ such that 
\begin{equation}\label{union7}
\bigcup_{i =0}^n \widehat{\tau}^k(U) = [0,1] 
\end{equation}
Here $\widehat{\tau}$ is the possibly multivalued map on $[0,1]$
which associates to each $x \in [0,1]$ the left and right hand limits
of $\tau$ at $x$. By construction this union is either $\{\tau(x)\}$ or
$\{1, 0\}$. In the latter case $0 = \tau(x)$. It follows therefore
that $\widehat{\tau}(A) \backslash \{1\}  = \tau(A)$ for every subset
$A \subseteq [0,1]$. Thus
$$
\widehat{\tau}^k(U) \supseteq \tau^k(U) 
$$
for all $k$. The strong transitivity of $\phi$ implies that
$\bigcup_{i=0}^{n-1} \tau^k(U) = [0,1[$ for some $n \in \mathbb N$. As observed above $\tau$ is
discontinuous at a point in $]0,1[$. It follows therefore that
$1 \in \widehat{\tau}\left([0,1[\right)$ and hence that (\ref{union7})
holds since
$$
\bigcup_{i=0}^n \widehat{\tau}^i(U) \supseteq \widehat{\tau} \left(
  \bigcup_{i=0}^{n-1} \widehat{\tau}^i(U)\right) \supseteq \widehat{\tau} \left(
  \bigcup_{i=0}^{n-1} \tau^i(U)\right) = \widehat{\tau}([0,1[) = [0,1] .
$$   

It follows now from Propositions 4.3 and 3.6 in \cite{S} that there is
a homeomorphism $h : [0,1] \to [0,1]$ such that $f = h \circ \tau
\circ h^{-1}$ is uniformly piecewise linear. From the proof of
Proposition 3.6 in \cite{S} we see that $h$ is increasing. Since
$\phi$ is not locally injective 
there
are non-empty open intervals $I,I' \subseteq \mathbb T \backslash
\{1\}$ such that $I \cap I' = \emptyset$ and $\phi(I) = \phi(I')$. Then $J = \mu(I)$ and $J'=
\mu(I')$ are non-empty open intervals in $[0,1[$ such that $J \cap J'
= \emptyset$ and $\tau(J) =
\tau(J')$, i.e. $\tau$ is not essentially injective in the sense of
Definition 4.1 of \cite{S}. Hence the slope $s$ of the linear pieces of
$f$ is $> 1$ by Proposition 4.3 of \cite{S}. 

Since $h(0) = 0, h(1) = 1$ and $\tau(0) =  \tau(1)$ we find that $f(0)
= f(1)$ and we can therefore define $\varphi : \mathbb T \to \mathbb
T$ such that $\varphi\left(e^{2 \pi it}\right) = e^{2 \pi i f(t)}, t
\in [0,1]$. Then $\varphi = g \circ \phi \circ g^{-1}$ where $g =
\mu^{-1} \circ h \circ \mu$. Then $g(1) = 1 = \lim_{\lambda \to1} g(\lambda)$.
Hence $g$ is continuous and an orientation preserving homeomorphism on $\mathbb T$. It follows
that $\varphi$ is a continuous map on $\mathbb T$ and conjugate to
$\phi$. By construction $\varphi$ is uniformly piecewise linear with
slope $s > 1$.
\end{proof}

We say that a $p$-periodic point $x  \in \mathbb T$ is \emph{repelling}
when there is an open interval $I$ in $\mathbb T$ and a $r > 1$ such
that $x \in I$ and $\left|\phi^p(y) - x \right| \geq r \left|y-x\right|$
for all $y \in I$.

\begin{lemma}\label{15} Assume that $\phi$ is transitive and uniformly piecewise linear with slope $s > 1$. Then the periodic points of $\phi$ are dense in $\mathbb
  T$ and they are all repelling. 
\end{lemma}
\begin{proof} Since $\phi$ is transitive there is a point in
  $\mathbb T$ with dense forward orbit, cf. Theorem 5.9 in
  \cite{W}. It follows therefore from Corollary 2 in \cite{AK} that
  $\phi$ has periodic points, and then by Corollary 3.4
  in \cite{CM} that the periodic points are dense. For each $n \in
  \mathbb N$ the map $\phi^n$ is uniformly piecewise linear with slope $s^n >
  1$. Therefore all periodic points of
$\phi$ are repelling.
\end{proof}

\begin{cor}\label{16} Assume that $\phi$ is transitive. It follows that $\Gamma_{\phi}^+$ is locally contractive
  in the sense of \cite{An}.
\end{cor}
\begin{proof} By Theorem \ref{11} we may assume that $\phi$ is
  piecewise linear with slope $s > 1$. Let $U$ be an open non-empty
  subset of $\mathbb T$. By Lemma \ref{15} there is in $U$ a point
  which is periodic and repelling. Since there are only finitely many
  critical points there are also only finitely many periodic orbits
  which contain a critical point. Hence $U$ contains a periodic point $x$,
  say of period $n$, which is repelling and whose orbit does not
  contain a critical point. Then $\val\left(\phi^{2n},x\right) =
  (+,+)$ and there is therefore an open neighborhood $W \subseteq U$ of $x$ and a
  $\kappa > 1$ such
  that $\val\left( \phi^{2n}, y\right) = (+,+)$ and $\left|\phi^{2n}(y) -x\right|
  \geq \kappa \left|y-x\right|$ for all $y \in W$. The proof is then
  completed exactly as the proof of Proposition 4.1 in \cite{Th4}.
\end{proof}

 \begin{lemma}\label{18} Assume that $\phi$ is transitive. It follows that $\Gamma^+_{\phi}$ is essentially free in
  the sense of \cite{An}, i.e. the points in $\mathbb T$ with trivial
  isotropy group in $\Gamma^+_{\phi}$ are dense in $\mathbb T$.
\end{lemma}
\begin{proof} A point in $\mathbb T$ has non-trivial isotropy group only
  when it is pre-periodic. It suffices therefore to show that the set of
  pre-periodic points has empty interior in $\mathbb T$; a fact which
  follows easily from the assumed transitivity of $\phi$.
%
\end{proof}

\begin{prop}\label{17} Assume that $\phi$ is transitive. It follows that $C^*_r\left(\Gamma_{\phi}^+\right)$ is
  purely infinite in the sense that every non-zero hereditary
  $C^*$-subalgebra contains an infinite projection.

\end{prop}
\begin{proof} This follows from Lemma \ref{18} and Corollary \ref{16},
  thanks to Proposition 2.4 in \cite{An}.
\end{proof}

There is one more fact about piecewise monotone circle maps
which can be deduced from the work of Shultz in \cite{S}, and which we shall use below. Recall that a continuous map $h : X \to X$ on a
compact Hausdorff space $X$ is \emph{totally transitive} when $h^n$ is
transitive for all $n \in \mathbb N$, and \emph{exact} when for all open
non-empty subsets $U \subseteq X$ there is an $N \in \mathbb N$
such that $h^N(U) = X$.

\begin{lemma}\label{schultz22} $\phi$ is exact if and only if $\phi$ is totally transitive.
\end{lemma} 
\begin{proof} It is obvious that exactness implies total transitivity. To prove
  the converse, return to the notation introduced in the proof of
  Theorem \ref{11} and assume that $\phi$ is totally transitive. As
  observed in that proof, $\tau$ is then transitive and not
  essentially injective in the sense of \cite{S}. It follows therefore
  from Corollary 4.7 in \cite{S} that there is an $N \in \mathbb N$ and
  closed sets $K_i, i = 1,2, \dots, N$, with mutually disjoint
  non-empty interiors $\Int K_i$ in $[0,1]$ such that $\tau^N$ maps
  $\Int K_i$ onto $\Int K_i$ and is exact on $K_i$ for each $i$. The
  image of $\Int K_i$ in $\mathbb T$ is open and invariant under
  $\phi^N$ and must therefore be all of $\mathbb T$ since $\phi$ is
  totally transitive. This implies that $N =1$, which
  means that $\tau$ is exact. It follows that $\phi$ is exact as well. 
\end{proof}

\section{Simplicity}

For $x \in \mathbb T$, let $\Ro(x)$ be the $\Gamma^+_{\phi}$-orbit of
$x$, i.e.
$$
\Ro(x) = \left\{y \in \mathbb T : \ \phi^n(x) = \phi^m(y), \
  \val\left(\phi^n,x\right) = \val\left(\phi^m, y\right) \ \text{for
    some} \ n,m \in \mathbb N\right\} .
$$
A subset $A \subseteq \mathbb T$ will be called \emph{restricted orbit
  invariant} or \emph{$\Ro$-invariant} when $x \in A \Rightarrow \Ro(x)
\subseteq A$. Let $Y \subseteq \mathbb T$ be a closed $\Ro$-invariant subset. Then
the reduction
$$
\Gamma^+_{\phi}|_Y = \left\{ (x,k,y) \in \Gamma^+_{\phi} : \ x,y \in Y
\right\}
$$
is a closed subgroupoid of $\Gamma^+_{\phi}$ and an \'etale groupoid
in the topology inherited from $\Gamma^+_{\phi}$. The same is true for
the reduction
$$
\Gamma^+_{\phi}|_{\mathbb T \backslash Y} = \left\{ (x,k,y) \in
  \Gamma^+_{\phi} : \ x,y \in \mathbb T \backslash Y
\right\} .
$$
The restriction map $C_c\left(\Gamma^+_{\phi}\right) \to
C_c\left(\Gamma^+_{\phi}|_Y\right)$ extends to a $*$-homomorphism
$\pi_Y : C^*_r\left(\Gamma^+_{\phi}\right) \to
C^*_r\left(\Gamma^+_{\phi}|_Y\right)$ and the inclusion
$C_c\left(\Gamma^+_{\phi}|_{\mathbb T \backslash Y}\right) \subseteq
C_c\left(\Gamma^+_{\phi}\right)$ extends to an embedding $C^*_r\left(\Gamma^+_{\phi}|_{\mathbb T \backslash Y}\right) \subseteq
C^*_r\left(\Gamma^+_{\phi}\right)$ which realizes
$C^*_r\left(\Gamma^+_{\phi}|_{\mathbb T \backslash Y}\right)$ as an
ideal in $C^*_r\left(\Gamma^+_{\phi}\right)$. It is straightforward to adopt
the proof of Lemma 3.2 in
\cite{Th3} to obtain the following.

\begin{lemma}\label{ideal} Let $Y$ be a closed $\Ro$-invariant subset
  of $\mathbb T$. It follows that
\begin{equation*}
\begin{xymatrix}{ 
0 \ar[r] & C^*_r\left(\Gamma^+_{\phi}|_{\mathbb T \backslash Y}\right)
\ar[r] & C^*_r\left(\Gamma^+_{\phi}\right) \ar[r]^-{\pi_Y} &
C^*_r\left(\Gamma^+_{\phi}|_Y\right) \ar[r] & 0}
\end{xymatrix}
\end{equation*} 
is exact.
\end{lemma}

In particular, $C^*_r\left(\Gamma^+_{\phi}\right)$ is not simple when
there are non-trivial closed $\Ro$-invariant subsets of
$\mathbb T$. We aim now to show that this is the only obstruction. 

For the statement of the next lemma recall that the \emph{full orbit} of a
point $x \in \mathbb T$ is the set 
$\left\{ y \in \mathbb T : \ \phi^n(y) = \phi^m(x) \ \text{for some} \
  n,m \in \mathbb N \right\}$. For each $j \in \mathbb N$ we let $\mathcal C_j$ denote the critical
points of $\phi^j$, i.e.
$$
\mathcal C_j = \left\{t \in \mathbb T : \ \val\left(\phi^j,t\right)
  \in \left\{ (+,-), (-,+)\right\} \right\} .
$$ 
The elements of $\bigcup_{n=0}^{\infty}
\phi^{-n}(\mathcal C_1)$ are then the \emph{pre-critical} points.

\begin{lemma}\label{fullB} Assume that there is a point $x \in \mathbb
  T$ whose full orbit is dense in $\mathbb T$. It follows that there
  is a point in $\mathbb T$ which is neither pre-periodic nor
  pre-critical.
\end{lemma}
\begin{proof} Let $\Per_n$ be the set of points in $\mathbb T$ of
  minimal period $n$. Assume
  for a contradiction that
$$
\mathbb T = \bigcup_{n,k \in \mathbb N} \phi^{-k}\left( \Per_n \cup
    \mathcal C_1\right) .
$$
By the Baire category theorem this implies that there are $k,n \in
\mathbb N$ such that $\phi^{-k}\left(\Per_n \cup \mathcal C_1\right)$
contains a non-degenerate interval. Since $\phi^k$ is piecewise monotone
and $\mathcal C_1$ finite this implies that $\Per_n$ contains a
non-degenerate interval. Then $\Per_n$ also contains two non-empty
open intervals $I_+,I_-$ such that $\bigcup_{i=0}^n
\phi^i\left(I_+\right)$ and $\bigcup_{i=0}^n
\phi^i\left(I_-\right)$ are disjoint. It follows that  
\begin{equation}\label{snit}
\left(\bigcup_{j = 0}^{\infty} \phi^j\left(I_+\right)\right) \cap
\left(\bigcup_{j = 0}^{\infty} \phi^j\left(I_-\right)\right) =
\emptyset .
\end{equation}
By assumption there is a point $x$ with dense full orbit. Since
both $I_+$ and $I_-$ contain an element from this orbit it follows
that here is a $k \in \mathbb N$ such that 
$$
\phi^k(x) \in \left(\bigcup_{j = 0}^{\infty} \phi^j\left(I_+\right)\right) \cap
\left(\bigcup_{j = 0}^{\infty} \phi^j\left(I_-\right)\right) .
$$
This contradicts (\ref{snit}). 
\end{proof}

\begin{lemma}\label{simplicity} The $C^*$-algebra $C^*_r\left(\Gamma^+_{\phi}\right)$ is
  simple if and only of $\Ro(x)$ is dense in $\mathbb T$ for all $x
  \in \mathbb T$.
\end{lemma}
\begin{proof} Simplicity of
  $C^*_r\left(\Gamma^+_{\phi}\right)$ implies that $\Ro(x)$ is dense
  for all $x$ by Lemma \ref{ideal}. For the converse assume that $\Ro(x)$ is dense for all $x$. By Corollary
  2.18 in \cite{Th1} the simplicity of
  $C^*_r\left(\Gamma^+_{\phi}\right)$ will follow if we can show that
  not all points of $\mathbb T$ have non-trivial isotropy in
  $\Gamma^+_{\phi}$. Since a point with non-trivial isotropy is
  pre-periodic it suffices to show that not all points of $\mathbb T$
  are pre-periodic under $\phi$. This follows from Lemma \ref{fullB}.  
\end{proof}

\begin{lemma}\label{what?} Assume that
  $C^*_r\left(\Gamma^+_{\phi}\right)$ is simple. Then $\phi$ is
  transitive.
\end{lemma}
\begin{proof} Let $E
  \subseteq \mathbb T$ be closed, with non-empty interior and
  $\phi$-invariant in the sense that $\phi(E) \subseteq E$. By Theorem
  5.9 \cite{W} it suffices to show that $E = \mathbb T$. For each $n,m
  \geq 1$ set
$$
U_{n,m} = \left\{ x \in \mathbb T: \ \phi^n(x) = \phi^m(y), \
  \val\left(\phi^n,x\right) = \val\left(\phi^m,y\right) \ \text{for
    some} \ y \in \Int E \right\}
$$
where $\Int E$ is the interior of $E$. Note that $U_{n,m}$ is open and
non-empty and that $\bigcup_{n,m} U_{n,m}$ is $\Ro$-invariant. It
follows therefore from Lemma \ref{simplicity} that $\bigcup_{n,m}
U_{n,m} = \mathbb T$. By compactness there is an $N \in \mathbb N$
such that $\mathbb T = \bigcup_{n,m=1}^N U_{n,m}$. Since
$U_{n,m} \subseteq \phi^{-n}(E)$ we
find then that
$$
\mathbb T =
\phi^N(\mathbb T) \subseteq \phi^N\left(\bigcup_{n=1}^N
  \phi^{-n}(E)\right) \subseteq E.
$$  
\end{proof}

The converse of Lemma \ref{what?} is not true in general; transitivity of $\phi$
does not imply that $C^*_r\left(\Gamma^+_{\phi}\right)$ is simple. A
necessary and sufficient condition for simplicity of
$C^*_r\left(\Gamma^+_{\phi}\right)$ will be given in Theorem \ref{S}.



\subsection{Finite $\Ro$-orbits and quotients of $C^*_r\left(\Gamma^+_{\phi}\right)$}\label{finiteRo}

The elements in $\phi(\mathcal C_1)$ are \emph{the critical
  values} and the elements of $\bigcup_{n=1}^{\infty} \phi^n(\mathcal C_1)$ are the
\emph{post-critical points}. Note that a
critical point is pre-critical, but not necessarily post-critical.

\begin{lemma}\label{4a} Assume that $\phi$ is transitive. Let $A \subseteq
  \mathbb T$ be a non-empty $\Ro$-invariant subset which is not dense in
  $\mathbb T$. It follows that $A$ is finite and consists of points
  that are post-critical and not pre-critical.
\end{lemma} 
\begin{proof} By assumption there is an open non-empty interval $J
  \subseteq \mathbb T$ such that 
\begin{equation}\label{inter}
A \cap J = \emptyset.
\end{equation} 
By Corollary 4.2 of \cite{Y} $\phi$ is not only transitive, but
also strongly transitive. There is therefore an $N\in \mathbb N$ such that
\begin{equation}\label{3}
\bigcup_{i=0}^N\phi^i(J) = \mathbb T.
\end{equation}
If $x \in A$ and
  $\val\left(\phi^j,x\right) \in \left\{(+,-),(-,+)\right\}$ for some
  $j \geq 1$ we can choose $y \in J$ such that $\phi^k(y) = x$ for
  some $k \in \{1,2,\dots, N\}$. It
  follows from the composition table for $\bullet$ that
  $\val\left(\phi^{k+j },y\right) = \val(\phi^j,x) \bullet
  \val(\phi^k,y) = \val \left(\phi^j,x\right)$. Hence
  $y \in \Ro(x) \subseteq A$, contradicting (\ref{inter}). It follows that
$\val \left(\phi^j,x\right) \in \left\{(+,+),(-,-)\right\}$ for all $j
\in \mathbb N$ when $x \in A$; i.e $A$ consists of points that are
not pre-critical.

Since $\phi$ is
not locally injective there is a $z \in \mathbb T$ such that
$\val \left( \phi, z\right) \in \left\{(+,-),(-,+)\right\}$. Choose
$z_0 \in J$ and $k \in \{1,2,\dots, N\}$ such that $\phi^k(z_0) = z$ and note that $\val \left(
  \phi^{k+1},z_0\right) \in \left\{(+,-),(-,+)\right\}$. There are
therefore subintervals $J_+,J_-$ of $J$ such that
$\val\left(\phi^{k+1},y\right) = (+,+)$ when $y \in J_+, \
\val\left(\phi^{k+1},y\right) = (-,-)$ when $y \in J_-$, and
$\phi^{k+1}\left(J_+\right) = \phi^{k+1}\left(J_-\right)
\overset{def}{=} I$. Since $\phi$ is strongly transitive there is a $K \in \mathbb N$ such that
$\bigcup_{i=1}^K\phi^i\left(I\right) = \mathbb T$.  Set $M_i = I \cap
\mathcal C_i$.
Let $a
\in \mathbb T$ be a non-critical element, i.e. $\val (\phi,a) \in
\left\{(+,+), (-,-)\right\}$. Assume that $a \notin \bigcup_{i=1}^K\phi^i(M_i)$. We
claim that $\Ro(a) \cap J \neq \emptyset$. To see this note that there
is an $i \in \{1,2,\dots, K\}$ and a $y' \in I \backslash M_i$
such that $\phi^i(y') = a$. Then $\val\left( \phi^i,y'\right) \in
\left\{(+,+),(-,-)\right\}$ and there is also an element $y \in J_+ \cup
J_-$ such that $\phi^{k+1} (y) = y'$ and 
$\val\left( \phi^{i+k+2}, y\right) = \val\left(\phi^{i+1},y'\right) \bullet \val\left(\phi^{k+1},y\right)
= \val\left(\phi,a\right)$.
It follows that $y \in \Ro(a) \cap J$, proving the claim. 

The last two paragraphs show that $A \subseteq \bigcup_{i=1}^K\phi^i(M_i)$. This completes
the proof because $\bigcup_{i=1}^K\phi^i(M_i)$ is finite and consists of post-critical points.
\end{proof}

Call a point $x \in \mathbb T$ \emph{exposed} when $\Ro(x)$ is
finite. By Proposition \ref{simplicity} and Lemma \ref{4a} it is the possible presence of
exposed points which is the only obstruction for simplicity of
$C^*_r\left(\Gamma^+_{\phi}\right)$.

\begin{cor}\label{exposedcor} Assume that $\phi$ is transitive. Then $C^*_r\left(\Gamma^+_{\phi}\right)$ is
  simple if and only if there are no exposed points.
\end{cor}


\subsubsection{$\left|\deg \phi \right| \geq 2$}

\begin{lemma}\label{corabsx} Assume that $\phi$ is transitive
  and that $\left|\deg \phi\right| \geq
  2$. It follows that $\Ro(x)$ is dense in $\mathbb T$ for all $x \in
  \mathbb T$.
\end{lemma}
\begin{proof} Let $n \in \mathbb N$. By looking at the
  graph of a lift $f : [0,1] \to \mathbb R$ of $\phi^{2n}$ one sees that
  for any $ x \in \mathbb T$, the set
$$
A_n = \left\{ y \in \mathbb T : \phi^{2n}(y) = x, \
  \val\left(\phi^{2n},y\right) = (+,+) \right\} 
$$
contains at least $\deg \phi^{2n}$ elements. Since $A_n \subseteq
\Ro(x)$ we conclude that $\Ro(x)$ is infinite for all $x \in \mathbb
T$. It follows then from Lemma \ref{4a} that $\Ro(x)$ is dense for all $x$.
\end{proof}

\begin{prop}\label{2simpe} Assume that $\phi$ is transitive
  and that $\left|\deg \phi\right| \geq 2$. It follows that
  $C^*_r\left(\Gamma^+_{\phi}\right)$ is simple.
\end{prop}

\subsubsection{$\left|\deg \phi \right| = 1$}

Before we specialize to the case where the degree is $1$ or $-1$ we need a couple of more general facts.

\begin{lemma}\label{nottot} Assume that $\phi$ is transitive, but not
  totally transitive. It follows that there is a $p > 1$ and closed
  intervals $I_i, i = 0,1,2,\dots, p-1$, such that
\begin{enumerate}
\item[1)] $\phi(I_i) = I_{i+1}$ (addition mod $p$),
\item[2)] $I_i \cap \Int I_j = \emptyset, \ i \neq j$,
\item[3)] $\bigcup_{i=0}^{p-1} I_i = \mathbb T$,
\item[4)] $\phi^p|_{I_i}$ is totally transitive for each $i$.
\end{enumerate}
\end{lemma}
\begin{proof} This is a special case of Corollary 2.7 in \cite{AdRR}.
\end{proof}

Note that the number $p$ and the collection $\{I_0,I_1, \dots,
I_{p-1}\}$ of intervals in Lemma \ref{nottot} are unique. We will refer
to $p$ as \emph{the global period of $\phi$}, and say that it is 1
when $\phi$ is totally transitive. In the following we
denote the set of endpoints of the intervals $I_i$ from Lemma
\ref{nottot} by $\mathcal E$.

\begin{lemma}\label{mathcalE} Assume that $\phi$ is transitive but not
  totally transitive. Then
\begin{equation}\label{E15}
\phi^{-1}\left(\mathcal E\right) \backslash \mathcal C_1 = \mathcal E.
\end{equation}
\end{lemma}
\begin{proof} Assume for a contradiction that $e \in \mathcal E$, but
  $\phi(e) \notin \mathcal E$. There are then intervals $I_i, I_{i'},
  I_j$ as in Lemma \ref{nottot} such that $i \neq i', \ e \in I_i \cap
  I_{i'}$ and $\phi(e) \in \Int I_j$. By continuity of $\phi$ and
  condition 1) from Lemma \ref{nottot} it follows that $I_{i+1} \cap
  \Int I_j \neq \emptyset$ and $I_{i'+1} \cap \Int I_j \neq
  \emptyset$. Since $i+1 \neq i'+1$ this violates condition 2). Thus
\begin{equation}\label{E11}
\phi(\mathcal E) \subseteq \mathcal E.
\end{equation}
If $e \in \mathcal E$ is a critical point the images $I_{i+1} = \phi\left(I_i\right)$ and
$I_{i'+1} = \phi\left(I_{i'}\right)$ of the two intervals $I_i, I_{i'}$
containing $e$ will both have non-trivial intersection with the same
interval $I_j$ containing $\phi(e)$; contradicting 2) again. Hence
\begin{equation}\label{E12}
\mathcal E \cap \mathcal C_1 = \emptyset .
\end{equation} 

Consider then an element $x \in \phi^{-1}(\mathcal E)$ and assume that
$x \notin \mathcal C_1$. Let $I_i$ and $I_{i'}$ be the two intervals
among the intervals from Lemma \ref{nottot} which contain $\phi(x)$.
If $x \notin \mathcal E$ there is a third interval $I_j$ which
contains $x$ in its interior. Since $x$ is not critical it follows
that $\phi(I_j) = I_{j+1}$ has non-trivial intersection with both
$\Int I_i$ and $\Int I_{i'}$, contradicting 2) once more. Hence 
\begin{equation}\label{E13}
\phi^{-1}(\mathcal E) \backslash \mathcal C_1 \subseteq \mathcal E.
\end{equation}
This completes the proof since (\ref{E15}) is equivalent to
(\ref{E11}), (\ref{E12}) and (\ref{E13}).
 
\end{proof}

\begin{lemma}\label{kluu} Assume that $\phi$ is transitive but not
  totally transitive. When $\deg \phi =1$ the set $\mathcal E$ is a
  $p$-periodic orbit where $p$ is the global period of $\phi$, and $\val\left(\phi,x\right) = (+,+)$ for all $x \in \mathcal E$. When
  $\deg \phi = -1$ the global period of $\phi$ is 2 and $\mathcal E$
  consists of two distinct fixed points of valency $(-,-)$.
\end{lemma}
\begin{proof} Let $I_i, i = 0,1,\cdots, p-1$, be the intervals from
  Lemma \ref{nottot}, and let $e_i^-$, be the left endpoint of $I_i$, defined using the orientation of
  $\mathbb T$. When $\deg \phi = 1$ we see by looking at the graph of a lift of $\phi$ that $\val\left(\phi, e^-_i\right) = (+,+)$ and
  $\phi\left(e^-_i\right) = e^-_{i+1}$ (addition mod $p$). It follows that $\mathcal E = \Ro\left(e^-_0\right)$, and that this is also the (forward) orbit of $e_0^-$. When $\deg \phi = -1$ observe
  first $\phi$ has a fixed point $x$. This fixed point lies in one of
  the intervals $I_i$. Since $x$ also lies in $I_{i+1}$ and $I_{i+2}$
it follows that two of the intervals $I_i$, $I_{i+1}$ and $I_{i+2}$
must be the same, i.e. $p = 2$. By looking at the graph of a lift of
$\phi$ we see that $\mathcal E$ consists of two fixed points of valency $(-,-)$.
\end{proof}

\begin{lemma}\label{oooo} If $\deg
  \phi = 1$ there is for all $x \in \mathbb T$ an element $y \in
  \phi^{-1}(x)$ such that $\val \left(\phi, y\right) = (+,+)$. If $\deg
  \phi = -1$ there is for all $x \in \mathbb T$ an element $y \in
  \phi^{-1}(x)$ such that $\val \left(\phi, y\right) = (-,-)$.
\end{lemma}
\begin{proof} Look at the graph of a lift of $\phi$.
\end{proof}

\begin{lemma}\label{fin} Assume that $\deg \phi \in \{1,-1\}$. Then
  $\Ro(x)$ is infinite for all $x \in \mathbb T$ that are not periodic
  under $\phi$.
\end{lemma}
\begin{proof} Let $x \in \mathbb T$. It follows from Lemma \ref{oooo}
  that there are sequences $\{n_i\}$ in
  $\mathbb N$ and $\left\{x_i\right\}$ in $\mathbb T$ such that
  $\phi^{n_1}(x_1) = x$, $\phi^{n_i}(x_i) = x_{i-1}, \ i \geq 2$, and
  $\val\left(\phi^{n_i},x_i\right) = (+,+)$ for all $i$. Then $x_i \in
  \Ro(x)$ for all $i$. The set $\{x_i : i \in
  \mathbb N\}$ is infinite when $x$ is not periodic.
\end{proof}

\begin{lemma}\label{Eroinv} Assume that $\deg \phi \in \{-1,1\}$ and $\phi$ is transitive but not
  totally transitive. Then $\mathcal E$ is the set of exposed points
  for $\phi$.
\end{lemma}
\begin{proof} Let $e \in \mathcal E$ and $y \in \Ro(e)$. There are natural
numbers $i,j \in \mathbb N$ such that $\phi^i(e) = \phi^{j}(y)$ and
$\val \left( \phi^i, e\right) = \val \left(\phi^{j}, y\right)$. It
follows from (\ref{E15}) $\phi^j(y) = \phi^i(x) \in \mathcal E$ and that 
$\val \left(\phi^i,e\right) \in(\pm,\pm)$ since $e\in \mathcal
E$. This implies first that $\val \left(\phi, \phi^{k}(y)\right) \in
(\pm,\pm)$ for all $k \leq j-1$ and then that $\phi^{j-1}(y) \in \phi^{-1}(\mathcal
E)\backslash \mathcal C_1 = \mathcal E$. But then $\phi^{j-2}(y) \in  \phi^{-1}(\mathcal
E)\backslash \mathcal C_1 = \mathcal E$ and so on. After $j$ steps we
conclude that $y \in \mathcal E$. This shows that $\mathcal E$ is $\Ro$-invariant.

It remains to show that $\mathcal E$ contains all exposed
points. Assume therefore that $y_0$ is an exposed point. It follows from Lemma
\ref{fin} that all exposed points are periodic. Since they are also
post-critical by Lemma \ref{4a} and there are only finitely many
critical points it follows that there are only finitely many
exposed points. Let $m\in \mathbb N$ be an even number divisible by
the global period $p$ and by all the
periods of exposed points. Then $\phi^m(y_0) = y_0$. Furthermore, if
$z \in \phi^{-m}(y_0)$ and $\val \left(\phi^m,z\right) = (+,+)$ we see
that $z \in \Ro(y_0)$ and hence $z$ is exposed. By definition of $m$
this implies that $\phi^m(z) = z$, i.e. $z = y_0$. To see that there
can not be any $z \in \phi^{-m}(y_0)$ with $\val\left(\phi^m,z\right)
= (-,-)$ observe by looking at the graph of the lift of $\phi^m$, that since $\deg \phi^m = 1$ the existence of such a
$z$ would imply the existence of a $z' \in \phi^{-m}(y_0) \backslash
\{y_0\}$ with $\val \left(\phi^m,z'\right) = (+,+)$ which is
impossible as we have just seen. Now assume for a contradiction that
$y_0 \notin \mathcal E$. Then $y_0$ lies in the interior of one of the
intervals from Lemma \ref{nottot}, say $I_i$. We can then write $I_i$
as the union $I_i = J_1 \cup J_2$ of two closed non-degenerate
intervals such that $J_1 \cap J_2 = \{y_0\}$. As we have just seen an
element of $I_i \cap \left(\phi^{-m} (y_0) \backslash \{y_0\} \right)$
must be critical for $\phi^m$ and it follows therefore that
$\phi^m(J_1) = J_1$. This
contradicts the total transitivity of $\phi^p|_{I_i}$. 
\end{proof}

\begin{lemma}\label{tot} Assume that $\phi$ is totally transitive and
  that $\deg \phi \in \{-1,1\}$. It follows that there is at most a
  single exposed point, and it must be a fixed point $e$ such that
  $\phi^{-1}(e) \backslash \mathcal C_1 = \{e\}$. 
\end{lemma}
\begin{proof} Let $m$ be the same number as in the proof of Lemma
  \ref{Eroinv}. In that proof it was shown that 
\begin{equation}\label{ooou}
\phi^{-m}(y)
  \backslash \{y\} \subseteq \mathcal C_m
\end{equation} 
for every exposed point
  $y$. It follows that if there are two exposed points, say $e_1$ and
  $e_2$, we could write $\mathbb T = J_1 \cup J_2$ where $J_1$ and
  $J_2$ are non-degenerate closed intervals such that $J_1\cap J_2 =
  \{e_1,e_2\}$ and $\phi^{2m}(J_i) = J_i, i = 1,2$. This contradicts the
  assumed total transitivity of $\phi$. Therefore there is at most at single exposed point $e$, and it
  is fixed by $\phi^m$. When $\deg \phi = 1$ it follows from Lemma
  \ref{oooo} that there is an element $z \in \phi^{-1}(e)$ such that
  $\val \left(\phi, z\right) = (+,+)$. Then $z$ is exposed (since $z
  \in \Ro(e)$) and the uniqueness of $e$ implies that $z = e$,
  proving that $e$ is a fixed point for $\phi$.

To reach the same conclusion when $\deg \phi = -1$ it suffices to consider the
case where $\val \left(\phi, e\right) = (-,-)$. By Lemma \ref{oooo}
there are elements $z_1,z \in \mathbb T$ such that $\phi(z_1) = z, \
\phi(z) = e$ and $\val\left(\phi,z_1\right) = \val\left(\phi,z\right) =
(-,-)$. Then $z_1 \in \Ro(e)$ and hence $z_1 = e$ because $e$ is the
only exposed point. It follows that $z = \phi(e)$, i.e. $\phi^2(e) =
e$. Note that $\val\left(\phi,\phi(e)\right) = (-,-)$. We claim that 
\begin{equation}\label{uuh}
\phi^{-1}\left(\{e , \phi(e)\}\right) \backslash \mathcal C_1 =
\{e,\phi(e)\}.
\end{equation}
To show this let $x \in \phi^{-1}(e) \backslash \mathcal
C_1$. If $\val (\phi, x) = (+,+)$ we find that $x = e$ since $e$
is the only exposed point. If $\val\left(\phi, x\right) = (-,-)$ an
application of Lemma \ref{oooo} shows that $x =
\phi(e)$. Consider then an element $y \in \phi^{-1}(\phi(e)) \backslash \mathcal
C_1$. If $\val(\phi,y) = (-,-)$ it follows that $y \in \Ro(e)$ and
hence $y = e$ by uniqueness of $e$. If instead $\val(\phi,y) = (+,+)$
an application of Lemma \ref{oooo} shows that that $y =
\phi(e)$. Having established (\ref{uuh}) note that it implies that
$\phi(e)$ is exposed, whence equal to $e$.

To show that $\phi^{-1}(e) \backslash \mathcal C_1 = \{e\}$ we may assume
that $\deg \phi = 1$, since the other case follows from
(\ref{uuh}). Furthermore, it suffices to show that $\phi^{-1}(e)
\backslash \mathcal C_1 \subseteq \{e\}$ since exposed points are not
critical by Lemma \ref{4a}. Consider therefore an element $x \in \phi^{-1}(e)
\backslash \mathcal C_1$. If $\val\left(\phi,x\right) = (+,+)$ it follows
that $x\in \Ro(e)$ and hence $x$ is exposed. Since $e$ is the only
exposed points this shows that $x = e$. Assume then that
$\val(\phi,x) = (-,-)$. If $x \neq e$ a look at
the graph for a lift of $\phi$ shows that there is then also a point
$y \in \phi^{-1}(e) \backslash \{e\}$ with $\val (\phi, y) =
(+,+)$ which we have just seen is not possible. Hence $x = e$

\end{proof}

To formulate the next proposition we
call a point $e \in \mathbb T$ an \emph{exceptional fixed point} when
$\phi^{-1}(e) \backslash \mathcal C_1 = \{e\}$.

\begin{prop}\label{deg1} Assume that $\phi$ is transitive and that
  $\deg \phi \in \{-1,1\}$. Then $C^*_r\left(\Gamma^+_{\phi}\right)$
  is simple unless either
\begin{enumerate}
\item[1)] $\phi$ is not totally transitive, or
\item[2)] $\phi$ is totally transitive and there is an exceptional
  fixed point.
\end{enumerate}
In case 2) there is an extension
\begin{equation}\label{extI}
\begin{xymatrix}{
0 \ar[r] & B \ar[r] & C^*_r\left(\Gamma^+_{\phi}\right) \ar[r] &
C(\mathbb T) 
\ar[r] & 0 }
\end{xymatrix}
\end{equation} 
where $B$ is simple and purely infinite. When $\phi$ is not totally transitive and $\deg \phi = 1$ there is an extension
\begin{equation}\label{extII}
\begin{xymatrix}{
0 \ar[r] & B \ar[r] & C^*_r\left(\Gamma^+_{\phi}\right) \ar[r] &
C(\mathbb T) \otimes M_p(\mathbb C)
\ar[r] & 0, }
\end{xymatrix}
\end{equation} 
where $p$ is the global period of $\phi$ and $B$ is simple and purely
infinite. When $\phi$ is not totally transitive and $\deg \phi = -1$ there is an extension
\begin{equation}\label{extIII}
\begin{xymatrix}{
0 \ar[r] & B \ar[r] & C^*_r\left(\Gamma^+_{\phi}\right) \ar[r] &
C(\mathbb T) \oplus C(\mathbb T) 
\ar[r] & 0, }
\end{xymatrix}
\end{equation} 
where $B$ is simple and purely infinite.

\end{prop}
\begin{proof} Assume that none of the two cases 1) or 2) occur. It
  follows from Lemma \ref{tot} that there are no exposed points,
  and from Corollary \ref{exposedcor} that
  $C^*_r\left(\Gamma^+_{\phi}\right)$ is simple.

In case 2) it follows from Lemma \ref{tot} that there is exactly one
exposed point, $e$, which is a fixed point. From Lemma \ref{ideal} we
get then the extension
\begin{equation}\label{ext77}
\begin{xymatrix}{
0 \ar[r] & B \ar[r] & C^*_r\left(\Gamma^+_{\phi}\right) \ar[r] &  C^*_r\left(\Gamma^+_{\phi}|_{\{e\}}\right)
\ar[r] & 0 }
\end{xymatrix}
\end{equation}
where $B = C^*_r\left(\Gamma^+_{\phi}|_{\mathbb T \backslash
    \{e\}}\right)$. It is easy to see, cf. the proof of Lemma 4.11 in
\cite{Th3}, that $C^*_r\left(\Gamma^+_{\phi}|_{\{e\}}\right) \simeq
C(\mathbb T)$. Furthermore, $B$ is purely infinite because $B$ is an ideal in
$C^*_r\left(\Gamma^+_{\phi}\right)$ which is purely infinite by
Proposition \ref{17}. To conclude that $B$
is simple we argue as in the proof of Proposition 4.10 in \cite{Th3}:
The elements of $\mathbb T\backslash \{e\}$ with non-trivial isotropy
in $C^*_r\left(\Gamma^+_{\phi}|_{\mathbb T \backslash \{e\}}\right)$
are pre-periodic. It follows from Theorem \ref{11} that the
pre-periodic points are countable, whence $\mathbb T\backslash \{e\}$
must contain a point with trivial isotropy. By Corollary 2.18 of
\cite{Th1} it suffices therefore to show that $\mathbb T
\backslash \{e\}$ does not contain any non-trivial (relatively)
closed $\Ro$-invariant subsets. Let therefore $L$ be such a set. Then
$L \cup \{e\}$ is closed and $\Ro$-invariant in $\mathbb T$
and hence either equal to $\mathbb T$ or contained in $\{e\}$ by Lemma
\ref{4a} and Lemma \ref{tot}. It follows that $ L = \emptyset$ or $L =  \mathbb T
\backslash \{e\}$. This completes the proof in case 2).

In case 1) we argue as above, except that we use Lemma \ref{Eroinv} to
replace Lemma
\ref{tot}, and Lemma \ref{kluu} to determine $C^*_r\left(\Gamma^+_{\phi}|_{\mathcal E}\right)$. 
\end{proof}

To show by example that all the cases mentioned in Proposition
\ref{deg1} can occur consider the graph

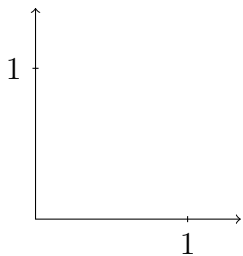
\begin{center}
\begin{tikzpicture}[x=2cm,y=2cm]
  \draw[<->] (0,1.4) -- (0,0) -- (1.35,0);
  \draw (1,0.02) --++ (0,-0.04) node [below] {$1$};
  \draw (0.02,1) --++ (-0.04,0) node [left] {$1$};
  \draw[thick,line join=round] plot file {z.dat};  
\end{tikzpicture}
\end{center}
The graph describes the lift of an exact, and hence totally transitive
map $\phi$ of the circle of degree $1$ with an exceptional  fixed point. The corresponding $C^*$-algebra
$C^*_r\left(\Gamma^+_{\phi}\right)$ is an extension as in
(\ref{extI}). To show that also the extensions (\ref{extII}) occur consider the
graph
\bigskip
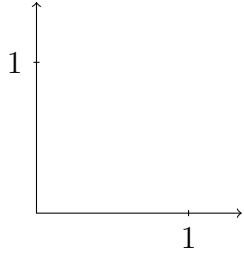
\begin{center}
\begin{tikzpicture}[x=2cm,y=2cm]
  \draw[<->] (0,1.4) -- (0,0) -- (1.35,0);
  \draw (1,0.02) --++ (0,-0.04) node [below] {$1$};
  \draw (0.02,1) --++ (-0.04,0) node [left] {$1$};
  \draw[thick,line join=round] plot file {x.dat};  
\end{tikzpicture}
\end{center}
This is the graph of the lift of a transitive, but not totally
transitive circle map $\phi$ of degree $1$ for which
$C^*_r\left(\Gamma^+_{\phi}\right)$ is an extension as in
(\ref{extII}) (with $ p = 4$). In the same way the following graph describes a transitive circle map
of degree $-1$ which is not totally transitive and for which $C^*_r\left(\Gamma^+_{\phi}\right)$ is an extension as in (\ref{extIII}).

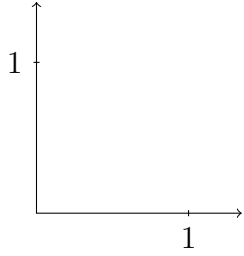
\begin{center}
\begin{tikzpicture}[x=2cm,y=2cm]
  \draw[<->] (0,1.4) -- (0,0) -- (1.35,0);
  \draw (1,0.02) --++ (0,-0.04) node [below] {$1$};
  \draw (0.02,1) --++ (-0.04,0) node [left] {$1$};
  \draw[thick,line join=round] plot file {w.dat};  
\end{tikzpicture}
\end{center}

\subsubsection{$\deg \phi = 0$}

A point $z \in \mathbb T$ will be called \emph{an exceptional critical
  value} when $\phi^{-1}(z)  \subseteq \mathcal C_1$.

\begin{lemma}\label{ecv1} Assume that
  $\deg \phi = 0$ and that $\phi$ is surjective. There is
  at most one exceptional critical value, and for all other elements $x \in \mathbb
  T$ there are points $y_{\pm} \in
  \phi^{-1}(x)$ such that $\val\left(\phi, y_{\pm}\right) = (\pm,\pm)$.
\end{lemma}
\begin{proof} Look at the graph of a lift of $\phi$.
\end{proof}

\begin{lemma}\label{ecv2} Assume that $\phi$ is transitive and that
  $\deg \phi = 0$. If $y \in \mathbb T$ is an exposed point there is an exceptional critical value $e\in \mathbb
  T$ such that $\phi^2(e) = \phi(e) \neq e$, $\Ro(y) =
\left\{e,\phi(e)\right\}$ and $\phi^{-1}\left(\phi(e)\right)
\backslash \mathcal C_1 = \left\{e, \phi(e)\right\}$.
\end{lemma}
\begin{proof} The main part of the proof will be to show that there is
  an exceptional critical value $e$ such that one of the following holds:  
\begin{enumerate}
\item[i)] $\phi^2 (e) = \phi(e) \neq e$, $\val\left(\phi, e\right) = (-,-)$ and $\Ro(y) =\{\phi(e)\}$,
\item[ii)]   $\phi^2(e) = \phi(e) \neq e$, $\val\left(\phi,
    \phi(e)\right) = (+,+)$ and $\Ro(y) =
\left\{e\right\}$,
\item[iii)]  $\phi^2(e) = \phi(e) \neq e$, $\Ro(y) =
\left\{e,\phi(e)\right\}$.
\end{enumerate}

Assume first that $\Ro(y)$ does not contain an exceptional critical value. Let $z \in
  \Ro(y)$. By using Lemma \ref{ecv1} we can
  construct $y_k, k = 0,1,2,3, \dots $ such that $y_0 = z$, $\phi(y_k)
  = y_{k-1}$ and $\val\left(\phi, y_k\right) = (+,+), \ k \geq 1$. Then $y_k \in
  \Ro(y)$ for all $k$ so there are $k \neq k'$ such that $y_k =
  y_{k'}$. It follows that $z$ is periodic and that
  $\val\left(\phi, u\right) = (+,+)$ for all $u$ in the orbit
  $\orb(z)$ of
  $z$. Hence $\Orb(z) \subseteq \Ro(y)$. Since this conclusion holds
  for all $z \in \Ro(y)$ and since the forward orbits of elements from
  $\Ro(y)$ must intersect we conclude that $\Ro(y) = \Orb(y)$ and
  $\val\left(\phi, \phi^k(y)\right) = (+,+)$ for all $k \in \mathbb
  N$. Let $z \in \Ro(y)$. Using Lemma \ref{ecv1} again we
find $u_1,v_1 \in \phi^{-1}(z)$ such that $\val\left(\phi, u_1\right)
= (+,+)$ and $\val\left(\phi, v_1\right) = (-,-)$. Then $u_1 \in
\Ro(y)$ and $u_1$ is therefore an element of the orbit
of $y$. Since $v_1 \neq u_1$ (or since $\val\left(\phi,v_1\right) = (-,-)$), it follows that $v_1$ is not in the orbit
of $y$. If $v_1$ is not an exceptional critical value we can find $v_2 \in \phi^{-1}(v_1)$
such that $\val\left( \phi, v_2\right) = (-,-)$. It follows that $v_2
\in \Ro(y)$ and $v_2$ must therefore be an element of $\Orb(y)$. This
contradicts that $v_1$ is not, and we conclude that $v_1$ must be an
exceptional critical value $e$, which by Lemma \ref{ecv1} is unique. This shows that $z = \phi(e) $ and we conclude therefore that case i) occurs.


We consider then the case where $\Ro(y)$ contains an exceptional
critical value $e$. By looking at the graph of a lift of $\phi$ we see
that a non-critical exceptional critical value $e$ can not be fixed since the
degree is $0$. Thus $\phi(e) \neq e$ since exposed points are not critical. To see that $\phi(e)$ is a fixed point assume that
  it is not. Consider first the case where $\phi(e)$ is periodic, say
  of period $p > 1$. Since $\phi(e) \neq e$ it follows from Lemma
  \ref{ecv1} that there is a point $b_1 \in \phi^{-1}(\phi(e))$ such that
  $\val \left( \phi,b_1\right) \neq \val\left(\phi,e\right)$. Then
  $b_1 \notin \{e,\phi(e)\}$ and we use Lemma \ref{ecv1} again to find
  $b_2 \in \phi^{-1}(b_1)$ such that $\val\left(\phi, b_2\right) \neq
  \val \left(\phi, \phi(e)\right)$. It follows that $b_2 \notin
  \left\{e,\phi(e),b_1\right\}$. By requiring in each step that
  $\val\left(\phi, b_i\right) \neq \val\left(\phi,\phi(e)\right)$ we obtain
  through repeated application of Lemma
  \ref{ecv1} elements $b_i, i = 1,2, \dots, p+1$, such that
  $\phi\left(b_{k+1}\right) = b_k$ and $b_{k+1} \notin
  \left\{e,\phi(e), b_1,b_2, \dots, b_k\right\}$ for all $k = 1,2,
  \dots, p$. Then, for $j > p+1$ we require in each step instead
  that $\val\left(\phi^{j}, b_j\right) = \val(\phi,e)$. It is then
  still automatic that $b_{k+1} \notin
  \left\{e, b_1,b_2, \dots, b_k\right\}$ for all $k$, while the
  fact that $b_j \neq \phi(e)$ follows for $j \geq p+1$ because $j$ is
  larger than the period of $\phi(e)$. Since $b_j \in \Ro(e) =
  \Ro(y)$ when $j > p+1$, we have contradicted the assume finiteness
  of $\Ro(y)$. To get the same contradiction when $\phi(e)$ is not
  assumed to be periodic we proceed in the same way, except that the
  steps between $b_1$ and $b_{p+1}$ can be bypassed. In any case we conclude that
  $\phi^2(e) = \phi(e)$. We next argue, in a similar way, that
$\phi^{-1}\left(\phi(e)\right) \backslash \mathcal C_1 \subseteq \{e,\phi(e)\}$. Indeed, if
$b_1 \in \phi^{-1}\left(\phi(e)\right) \backslash \left( \{e,\phi(e)\}
  \cup \mathcal C_1\right)$
we use Lemma \ref{ecv1} to get a sequence $b_i$ such that
$\phi\left(b_{i+1}\right) = b_i, i \geq 1$, and $\val
\left(\phi,b_i\right) = (-,-), i \geq 2$. Then $i \neq i' \Rightarrow
b_i \neq b_{i'}$, and $b_i \in \Ro(e)$ for infinitely many $i$; again
contradicting the infiniteness of $\Ro(e)$. Since $e$ is not pre-critical by Lemma \ref{4a} we have shown that
$\phi^{-1}\left(\phi(e)\right) \backslash \mathcal C_1= \{e,\phi(e)\}$. If $\val\left(\phi,
  e\right) = (-,-)$ and $\val\left(\phi,\phi(e)\right) = (+,+)$ we
find now easily that $\Ro(y) = \Ro(e) = \{e\}$, which is case ii), and in all other cases that
$\Ro(y) = \Ro(e) = \{e,\phi(e)\}$, which is case iii).

Finally we argue that the cases i) and ii) are
impossible. Indeed, in both cases we must have that $\val(\phi,e) =
(-,-)$ and $\val(\phi,\phi(e)) = (+,+)$ since otherwise $e \in
\Ro(\phi(e))$. But then the two closed intervals $J_1$ and
$J_2$ defined such that $J_1 \cap J_2 = \{e, \phi(e)\}$ and $J_1 \cup
J_2 = \mathbb T$ are both $\phi$-invariant, which contradicts the
transitivity of $\phi$. It follows that only case iii) can
occur. 

\end{proof}

\begin{lemma}\label{Atot} Assume that $\phi$ is transitive and $\deg
  \phi = 0$. Then there are exposed points if and only if $\phi$ is
  not totally transitive. 
\end{lemma}
\begin{proof} If $\phi$ is not totally transitive there are exposed
  points by (the proof of) Lemma \ref{Eroinv}. Conversely, if there are exposed
  points it follows from Lemma \ref{ecv2} that there is an exceptional
  critical value $e$ such that $e \neq \phi(e) = \phi^2(e)$ and $\{e,
  \phi(e)\}$ is the set of exposed points. Furthermore,
  $\phi^{-1}(\phi(e)) \backslash \mathcal C_1 = \{e, \phi(e)\}$. The points $e$ and
  $\phi(e)$ define closed intervals $J_1$ and $J_2$ such that $\mathbb
  T = J_1 \cup J_2$, $J_1 \cap J_2 = \{e, \phi(e)\}$ and $\phi(J_i) =
  J_i, \ i =1,2$, or $\phi(J_1) = J_2$ and $\phi(J_2) = J_1$. The first
  case is ruled out by transitivity, and the second implies that
  $\phi$ is not totally transitive.
\end{proof}

\begin{prop}\label{deg0} Assume that $\phi$ is transitive and that
  $\deg \phi = 0$. Then $C^*_r\left(\Gamma_{\phi}\right)$ is simple if
  and only if $\phi$ is totally transitive. When $\phi$ is not totally
  transitive there is an extension
\begin{equation}\label{ext2}
\begin{xymatrix}{
0 \ar[r] & B \ar[r] & C^*_r\left(\Gamma^+_{\phi}\right) \ar[r] &
C(\mathbb T) \otimes M_2(\mathbb C)
\ar[r] & 0 }
\end{xymatrix}
\end{equation} 
where $B$ is simple and purely infinite.
\end{prop}

\begin{proof} With Lemma \ref{Atot} and Lemma \ref{ecv2} at hand all the necessary
  arguments can be found in the proof of Proposition \ref{deg1}.

\end{proof}

The following graph describes a transitive circle map
of degree $0$ which is not totally transitive and for which $C^*_r\left(\Gamma^+_{\phi}\right)$ is an extension as in (\ref{ext2}).

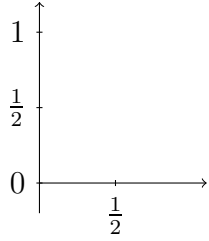
\begin{center}
\begin{tikzpicture}[x=2cm,y=2cm]
  \draw[->] (0,-0.5) -- (1.1,-0.5);
  \draw[->] (0,-0.7) -- (0,0.7);
  \draw (0.02,0.5) --++ (-0.04,0) node [left] {$1$};
  \draw (0.02,0) --++ (-0.04,0) node [left] {$\tfrac12$};
  \draw (0.02,-0.5) --++ (-0.04,0) node [left] {$0$};
  \draw (0.5,-0.48) --++ (0,-0.04) node [below] {$\tfrac12$};
  \draw[thick,line join=round] plot file {y.dat};   
\end{tikzpicture}

\end{center}

\subsubsection{Simplicity}

We can now finally give a necessary and sufficient condition for
simplicity of $C^*_r\left(\Gamma^+_{\phi}\right)$.

\begin{thm}\label{S} The following conditions are equivalent.
\begin{enumerate}
\item[i)] $C^*_r\left(\Gamma^+_{\phi}\right)$ is simple.
\item[ii)] $\phi$ is totally transitive and there is no exceptional
  fixed point.
\item[iii)] $\phi$ is exact and there is no exceptional fixed point.
\end{enumerate}
\end{thm} 
\begin{proof} i) $\Rightarrow$ ii) : Transitivity of $\phi$ follows
  from Lemma \ref{what?}. And then $\phi$ must be totally transitive
  since otherwise the set $\mathcal E$ considered in Lemma
  \ref{mathcalE} will be non-empty, finite and $\Ro$-invariant, as
  shown in the first paragraph of the proof of Lemma \ref{Eroinv}, and
  this contradicts simplicity by Lemma \ref{simplicity}. The absence
  of an exceptional fixed point follows also from Lemma
  \ref{simplicity} since an exceptional fixed point is its own
  $\Ro$-orbit. 

The implication ii) $\Rightarrow$ i) follows from Propositions
  \ref{deg0}, \ref{deg1} and \ref{2simpe}, and the implication iii)
  $\Rightarrow$ ii) is trivial. The implication ii) $\Rightarrow$ iii) follows from Lemma \ref{schultz22}.
\end{proof}

\section{Nuclearity, UCT and a six-terms exact sequence}

\subsection{The Cuntz-Pimsner picture of $C^*_r\left(\Gamma^+_{\phi}\right)$}

To simplify notation, set
$$
R^+_{\phi} = \Gamma^+_{\phi}(0) .
$$
Then $C^*_r\left(R^+_{\phi}\right)$ is the fixed point algebra of \emph{the
gauge action} $\beta = \beta^c$ on $C^*_r\left(\Gamma^+_{\phi}\right)$
arising from the homomorphism $c : \Gamma^+_{\phi} \to
\mathbb Z$ defined such that $c(x,k,y) = k$, cf. \cite{Re}. For $n \in \mathbb N$, set 
$$
R^+_{\phi}(n) = \Gamma^+_{\phi}(0,n)
$$
which is an open sub-groupoid of $R^+_{\phi}$. Then $R^+_{\phi} =
\bigcup_n R^+_{\phi}(n)$ and
\begin{equation}\label{union}
C^*_r\left(R^+_{\phi}\right) = \overline{\bigcup_n
  C^*_r\left(R^+_{\phi}(n)\right)} .
\end{equation}

\begin{lemma}\label{cruX} Assume that
  $C^*_r\left(\Gamma^+_{\phi}\right)$ is simple. Let $x \in
  \mathbb T \backslash \mathcal C_1$. It follows that there are
  elements $z,z' \in \mathbb T\backslash \mathcal C_1$ such that
  $(x,1,z), \ (z',1,x) \in \Gamma^+_{\phi}$.
\end{lemma}
\begin{proof} If $\phi^k(x) \in \mathcal C_1$ for some $k \geq 1$, set
  $z = \phi(x)$ and let $z'$ be any element of $\phi^{-1}(x)$. Then
  $(x,1,z), \ (z',1,x) \in \Gamma^+_{\phi}$. Assume
  therefore now
  that $\phi^n(x) \notin \mathcal C_1$ for all $n \geq 1$. Since
  $\phi$ is not locally injective there are open non-empty intervals
  $I_{\pm}$ and $I$ such that $\val(\phi,z) = (+,+)$ for all $z \in
  I_+$, $\val(\phi,z) = (-,-)$ for all $z \in I_-$ and $\phi(I_-) =
  \phi(I_+) = I$. It follows from Theorem \ref{S} that $\phi$ is exact and there is therefore an $N \in \mathbb N$
  such that $\phi^{N-1}(I) = \mathbb T$. 

We consider
  first the case where $x$ is pre-periodic
 to a finite orbit
  $\mathcal O$ of period $p$. Let $d \in \{-1,1\}$. If there is an $M \in \mathbb N$ such that
  $\phi^{-j}(\mathcal O) \subseteq \mathcal C_j \cup \bigcup_{k
    =0}^{j-1}\phi^{-k}(\mathcal O)$ for all $j > M$, it
  follows that $\Ro(x) \subseteq \bigcup_{j=0}^M \phi^{-j}(\mathcal
  O)$ which is a finite set. This is impossible since
  $C^*_r\left(\Gamma^{+}_{\phi}\right)$ is simple, cf. Corollary
  \ref{exposedcor}. Since $\mathcal O$ is finite it follows that there
  is a $z \in \mathcal O$ and an $n_0 \in \mathbb N$ such that
  $\phi^{-j}(z) \backslash (\mathcal C_j \cup \bigcup_{k=0}^{j-1}\phi^{-k}(\mathcal O))  \neq \emptyset$ for all $j
  \geq n_0$. Since $z \in \mathcal O$ there are $k,l \in \mathbb N$
  such that $\phi^l(z) =
  \phi^{kpN+d}(x)$. Note that 
$$
j \neq j'  
  \Rightarrow \left(\phi^{-j}(z) \backslash (\mathcal C_j \cup
    \bigcup_{k=0}^{j-1}\phi^{-k}(\mathcal O)) \right) \cap \left(\phi^{-{j'}}(z) \backslash
    (\mathcal C_{j'} \cup \bigcup_{k=0}^{j'-1}\phi^{-k}(\mathcal
    O))\right) = \emptyset .
$$ 
Since $\phi^N(\mathcal C_N)$ is a finite set there is therefore a $k' \geq k$ such that $k'pN-N-l \geq
  n_0$ and
$$
\phi^N\left(\mathcal C_N\right) \cap \left(\phi^{-k'pN + N+l}(z) \backslash
\mathcal C_{k'pN-N-l}\right)  = \emptyset.
$$
Choose an element $a \in \phi^{-k'pN + N+l}(z) \backslash
\mathcal C_{k'pN-N-l}$. Then $a = \phi^N(y_{\pm})$
 where
$y_{\pm} \in I_+ \cup I_-$ and $\val(\phi^N,y_+) =
(+,+), \ \val(\phi^N,y_-)  = (-,-)$. It follows that
$\phi^{k'pN}(y_{\pm}) = \phi^{kpN+d}(x) = \phi^{k'pN+d}(x)$ and that either $(x,d,y_+) \in \Gamma^+_{\phi}$ or $(x,d,y_-) \in
\Gamma^+_{\phi}$. When $d = 1$ this gives us $z$ such that $(x,1,z)
\in \Gamma^+_{\phi}$ and when $d = -1$ it gives us $z'$ such that
$(x,-1,z') \in \Gamma^+_{\phi}$.

Consider then the case where $x$ is not pre-periodic. If there are
infinitely many $n \in \mathbb N$, $n > N$, such that the only elements of
$\phi^{-N}\left(\phi^{n}(x)\right) \backslash \left\{
  \phi^{n-N}(x)\right\}$ are critical points for $\phi^N$, it
follows that $\phi^{n}(x) \in \phi^N\left(\mathcal C_N\right)$ for
infinitely many $n$. Since $\phi^N\left(\mathcal C_N\right)$ is a
finite set this contradicts that $x$ is not pre-periodic. It follows
that when $x$ is not pre-periodic there is an $n_0 \in \mathbb N$ such
that for all $n \geq n_0$ there is an element $y_n \in \phi^{-N}\left(\phi^{n}(x)\right)$ which is not in
the forward orbit $\left\{\phi^j(x) : j \in \mathbb N\right\}$ of $x$
and also not critical for $\phi^N$. Note that
$$
j  \neq j' \Rightarrow \left(\bigcup_{k \in \mathbb N}
  \phi^{-kN}\left(y_j\right) \right) \cap \left(\bigcup_{k \in \mathbb N}
  \phi^{-kN}\left(y_{j'}\right)\right) = \emptyset 
$$
when $j,j' \geq n_0$. Since $\phi^N(\mathcal C_N)$ is a finite set
there is therefore an $m  > 2$ such that $mN+d > n_0$ and 
$$
\left(\bigcup_{k \in \mathbb N}
  \phi^{-kN}\left(y_{mN+d}\right)\right) \cap \phi^N(\mathcal C_N) =
\emptyset.
$$  
Let $a \in \phi^{-(m-2)N}\left(y_{mN+d}\right)$. Choose $y_{\pm} \in
I_{\pm}$ such that $\phi^N(y_{\pm}) = a$. Then
$\phi^{mN}\left(y_{\pm}\right) = \phi^{mN +d}(x)$ and for one of
the elements $v$  in $\left\{y_+,y_-\right\}$ it holds that
$\val\left(\phi^{mN},v\right) = \val\left(\phi^{mN+d},x\right)$. Then
$(x,1,v) \in \Gamma^+_{\phi}$ when $d = 1$ and $(x,-1,v) \in
\Gamma^+_{\phi}$ when $d = -1$.

\end{proof}

\begin{lemma}\label{out} Assume that
  $C^*_r\left(\Gamma^+_{\phi}\right)$ is simple. Let $x \in \mathcal
  C_1$ and let $U$ be an open non-empty subset of $\mathbb T$. There are elements $\mu_1,\mu_2, \cdots, \mu_N \in
  \Gamma^+_{\phi}(1)$ such that $\gamma = \mu_1\mu_2 \mu_3 \cdots
  \mu_N$ is defined, $s(\gamma) =x$ and $r(\gamma) \in U \backslash
  \mathcal C_1$.
\end{lemma}
\begin{proof} Since $\phi$ is exact the backward orbit $\bigcup_{j
    =1}^{\infty}\phi^{-j}(x)$ is dense in $\mathbb N$.  so there is an $N \in
  \mathbb N$ and a $z \in \phi^{-N}(x) \cap U  \backslash \mathcal C_1$. Set
  $\mu_i = \left(\phi^{i-1}(z),1,\phi^i(z)\right)$.
\end{proof}

\begin{lemma}\label{specgen} Assume that
  $C^*_r\left(\Gamma^+_{\phi}\right)$ is simple. Then the first
  spectral subspace for the gauge action, i.e. the set
$$
E_1 = \left\{a \in C^*_r\left(\Gamma^+_{\phi}\right) : \
  \beta_{\lambda}(a) = \lambda a \ \forall \lambda \in \mathbb T
\right\},
$$
generates $C^*_r\left(\Gamma^+_{\phi}\right)$ as a $C^*$-algebra.
\end{lemma}
\begin{proof} The gauge action $\beta$ on $C^*_r\left(\Gamma^+_{\phi}\right)$
  restricts to an action on $C^*_r\left(\Gamma^+_{\phi}|_{\mathbb T
      \backslash \mathcal C_1}\right)$. Let $V_1 = E_1 \cap C^*_r\left(\Gamma^+_{\phi}|_{\mathbb T
      \backslash \mathcal C_1}\right)$ be the first spectral subspace of the restricted action. We claim
that
$V_1V_1^*$ spans a dense subspace in the fixed point algebra
$C^*_r\left(\Gamma^+_{\phi}|_{\mathbb T
      \backslash \mathcal C_1}\right)^{\beta}$.
To show this observe that since the closed span of $V_1V_1^*$ is an ideal in $C^*_r\left(\Gamma^+_{\phi}|_{\mathbb T
      \backslash \mathcal C_1}\right)^{\beta}$ it suffices to show
  that the span of $V_1V_1^*$ contains an approximate unit for $C^*_r\left(\Gamma^+_{\phi}|_{\mathbb T
      \backslash \mathcal C_1}\right)$. Hence it suffices to show that
  $C_c\left(\mathbb T \backslash \mathcal C_1\right) \subseteq
  \Span V_1V_1^*$. Let $f \in C_c\left(\mathbb T \backslash \mathcal
    C_1\right)$. Let $x \in \supp f$. It follows from Lemma \ref{cruX}  that there is a bisection $U \subseteq \Gamma^+_{\phi}(1)\cap
  \Gamma^+_{\phi}|_{\mathbb T \backslash \mathcal C_1}$ such that $x
  \in r(U)$. There is therefore a function $g\in C_c(U)$ such
  that $gg^*\in C_c(\mathbb T \backslash \mathcal C_1)$ and $gg^*(x) =
  1$. In this way we get a finite collection $g_i, i = 1,2,\dots, N$,
  in $C_c\left(\Gamma^+_{\phi}(1)\cap
  \Gamma^+_{\phi}|_{\mathbb T \backslash \mathcal C_1}\right)$ such
that $g_ig_i^* \in C_c\left(\mathbb T \backslash \mathcal C_1\right)$
for all $i$ and $\sum_{i=1}^N g_ig_i^*(y) > 0$ for all $y \in \supp
f$. There is then a function $h \in C_c\left(\mathbb T \backslash
  \mathcal C_1\right)$ such that $fh\sum_{i=1}^Ng_ig_i^* = f$. Since
$fhg_i$ and $g_i$ are elements of $V_1$ for all $i$ this shows that $f \in \Span V_1V_1^*$.  

A similar argument shows that also $V_1^*V_1$ spans a dense
subspace of $C^*_r\left(\Gamma^+_{\phi}|_{\mathbb T
      \backslash \mathcal C_1}\right)^{\beta}$. Thus the
  restriction of the gauge action is full on $C^*_r\left(\Gamma^+_{\phi}|_{\mathbb T
      \backslash \mathcal C_1}\right)$ and it follows that $C^*_r\left(\Gamma^+_{\phi}|_{\mathbb T
      \backslash \mathcal C_1}\right)$ is generated by $V_1$, whence $C^*_r\left(\Gamma^+_{\phi}|_{\mathbb T
      \backslash \mathcal C_1}\right)$ is contained in the
  $C^*$-subalgebra of $C^*_r\left(\Gamma^+_{\phi}\right)$ generated by
  $E_1$.

Note that it follows
from Lemma \ref{out} that there are elements $g_x, x \in \mathcal C_1$, in the $*$-algebra
generated by $C_c\left(\Gamma^+_{\phi}(1)\right)$ such that $g_{x'}^*g_x = 0$ when $x\neq x'$,
$g_x^*g_x \in C(\mathbb T)$, $g^*_xg_x(x) = 1$ and $g_x$ is supported
in $r^{-1}(\mathbb T \backslash \mathcal C_1)$ for all
$x$. Consider then an element $f \in
C_c\left(\Gamma^+_{\phi}\right)$. Write
\begin{equation*}\label{summands1}
f = f - \sum_{x \in \mathcal C_1} g_x^*g_xf +
\sum_{x \in \mathcal C_1} g_x^*g_xf  = f - \sum_{x \in \mathcal
  C_1} g_x^*g_xf + \sum_{x \in \mathcal C_1} g_x^* h,
\end{equation*}
where $h = \sum_{x \in \mathcal C_1} g_xf$. Note that $f - \sum_{x \in \mathcal C_1} g_x^*g_xf $ and $h$ are both supported in
$r^{-1}\left(\mathbb T \backslash \mathcal C_1\right)$. To conclude
that $f$ is contained in the $C^*$-algebra generated by $E_1$ we may
therefore assume that $f$ is supported in
$r^{-1}\left(\mathbb T \backslash \mathcal C_1\right)$. Under this
assumption we write
\begin{equation}\label{summands2}
f = f - \sum_{x \in \mathcal C_1} fg_x^*g_x +
\sum_{x \in \mathcal C_1} fg_x^*g_x ,
\end{equation}
and note that $f - \sum_{x \in \mathcal C_1} fg_{x}^*g_{x}$ and $\sum_{x \in
  \mathcal C_1} fg_{x}^*$ are elements of $C_c\left(\Gamma^+_{\phi}|_{\mathbb T
    \backslash \mathcal C_1}\right)$. It
follows therefore from the first part of the proof that both are
elements of the
$C^*$-algebra generated by $E_1$. Since 
$$
\sum_{x \in \mathcal C_1}
fg_x^*g_x = \sum_{x' \in \mathcal C_1} \left(\sum_{x \in \mathcal C_1}
  fg_x^*\right)g_{x'}
$$
 it follows from (\ref{summands2}) that $f$ is in the $C^*$-algebra
generated by $E_1$.
\end{proof}

\begin{lemma}\label{minus1} Assume that
  $C^*_r\left(\Gamma^+_{\phi}\right)$ is simple. It follows that
  $E_1^*E_1$ has dense span in $C^*_r\left(R^+_{\phi}\right)$.
\end{lemma}
\begin{proof} Since the closed span of $E_1^*E_1$ is an ideal in $C^*_r\left(R^+_{\phi}\right)$ it suffices to show
  that $E_1^*E_1$ contains $1 \in C(\mathbb T)$. Let $f \in C\left(\mathbb
    T\right)$. Let $x \in \mathbb T$. Provided $x \notin \mathcal C_1$
  it follows from Lemma \ref{cruX} that there is an element $\gamma
  \in \Gamma^+_{\phi}(1)$ such that $s(\gamma) = x$. When $x \in
  \mathcal C_1$ set $\gamma = (z,1,x)$, where $z \in
  \phi^{-1}(x)$. Then $s(\gamma) = x$. Hence, regardless of which $x$
  we consider there is a bisection $U \subseteq \Gamma^+_{\phi}(1)$ such that $x
  \in s(U)$. It follows that there is a function $g\in C_c(U)$ such
  that $g^*g\in C(\mathbb T)$ and $g^*g(x) =
  1$. In this way we get a finite collection $g_i, i = 1,2,\dots, N$,
  in $C_c\left(\Gamma^+_{\phi}(1)\right)$ such
that $g_i^*g_i \in C\left(\mathbb T\right)$
for all $i$ and $\sum_{i=1}^N g_i^*g_i > 0$. There is then a function $h \in C\left(\mathbb T\right)$ such that $h\sum_{i=1}^Ng_i^*g_i = 1$. Since
$g_ih^*$ and $g_i$ are elements of $E_1$ for all $i$, this completes
the proof.
\end{proof}

It should be observed that the gauge
action is generally not full. Even when
$C^*_r\left(\Gamma^+_{\phi}\right)$ is simple it can easily happen
that the ideal in $C^*_r\left(R^+_{\phi}\right)$ generated by
$E_1E_1^*$ is a proper ideal. Assume, for example, that $\phi$ has two critical points whose
  forward orbits do not intersect, and that one of them, say $c$,
  has infinite forward orbit. Then 
$r^{-1}(c) \cap \Gamma^+_{\phi}(1) = \emptyset$. Let $\omega$ be the state
on $C^*_r\left(R^+_{\phi}\right)$ such that $\omega(f) = f(c,0,c)$
when $f \in C_c\left(R^+_{\phi}\right)$. Then $\omega(fg^*) = 0$ for all
$f,g \in C_c\left(\Gamma^+_{\phi}(1)\right)$ and it follows that $\omega$
annihilates the ideal in $C^*_r\left(R^+_{\phi}\right)$ generated by $E_1E_1^*$. This asymmetry in the gauge action is responsible for some intriguing features of the KMS states.

Let $E_{-1} = E_1^*$ be the first negative spectral subspace. Then $E_{-1}$
is a $C^*_r\left(R^+_{\phi}\right)$-correspondence and we can consider
the associated Cuntz-Pimsner $C^*$-algebra $\mathcal O_{E_{-1}}$, cf. \cite{Ka}. We can
now show that this is another version of
$C^*_r\left(\Gamma^+_{\phi}\right)$ when the latter algebra is simple.

\begin{thm}\label{cunPi} Assume that
  $C^*_r\left(\Gamma^+_{\phi}\right)$ is simple. It follows that
  $C^*_r\left(\Gamma^+_{\phi}\right) \simeq \mathcal O_{E_{-1}}$ and
  there is an exact sequence
\begin{equation}\label{6termyes}
\begin{xymatrix}{
K_0\left( C^*_r\left(R^+_{\phi} \right)\right) \ar[rr]^-{\id -
  [E_{-1}]_0} & & K_0\left(
    C^*_r\left(R^+_{\phi} \right)\right) \ar[rr]^{\iota_0}  &
& K_0\left( C^*_r\left(\Gamma^+_{\phi}\right) \right)\ar[d] \\
K_1\left(C^*_r\left(\Gamma^+_{\phi}\right) \right)\ar[u]  & &
K_1\left(C^*_r\left(R^+_{\phi}\right) \right) \ar[ll]^{\iota_1}  & &  K_1\left(
  C^*_r\left(R^+_{\phi}\right)\right) \ar[ll]^-{\id -
 [E_{-1}]_1},
}\end{xymatrix}
\end{equation} 
where $\iota :
C^*_r\left(R^+_{\phi} \right) \to
C^*_r\left(\Gamma^+_{\phi}\right)$ is the inclusion map and 
$$
[E_{-1}] \in
KK\left(C^*_r\left(R^+_{\phi} \right), C^*_r\left(R^+_{\phi} \right)\right)
$$ 
is the KK-theory element represented by $E_{-1}$.
\end{thm}
\begin{proof} The inclusions $C^*_r\left(R^+_{\phi}\right) \subseteq
  C^*_r\left(\Gamma^+_{\phi}\right)$ and $E_{-1} \subseteq
  C^*_r\left(\Gamma^+_{\phi}\right)$ define a representation of the
  $C^*$-correspondence $E_{-1}$ in $C^*_r\left(\Gamma^+_{\phi}\right)$
  as defined by Katsura in Definition 2.1 of \cite{Ka}. Observe that
  it follows from Lemma \ref{minus1} that the representation is
  injective in the sense of Katsura, and that
\begin{equation}\label{cruxx}
 \psi_t\left(C^*_r\left(R^+_{\phi}\right)\right) = \mathcal K(E_{-1}),
\end{equation} 
in
    the notation from \cite{Ka}. It follows now from Proposition 3.3 in \cite{Ka} that our
representation of $E_{-1}$ is covariant in the sense of Definition 3.4 of
\cite{Ka}. Then the isomorphism $\mathcal O_{E_{-1}} \simeq
C^*_r\left(\Gamma^+_{\phi}\right)$ follows from Lemma \ref{specgen}
above and Theorem 6.4 in \cite{Ka}. Thanks to (\ref{cruxx}) we get now the
stated six-terms exact sequence from Theorem 8.6 in \cite{Ka}.

\end{proof}

\subsection{The structure of $C^*_r\left(R^+_{\phi}(k)\right)$ and
  some consequences}\label{extMar3II}

Fix a natural number $k$ and let $\mathcal D$ be a finite subset of $\mathbb T$ such
that $\mathcal C_k \subseteq \phi^{-k}(\mathcal D)$. Let $c_1
<c_2<
\cdots < c_{N}$ be a numbering of the elements in
$\mathcal D$ and set $c_0 = c_N$. Let $I_i =
\left]c_{i-1},c_{i}\right[, i = 1,2, \dots,N$. For each $i$ we fix a homeomorphism $\psi_{I_i} :
]0,1[ \to \left]c_{i-1},c_i\right[$ such that $\lim_{t \to 0} \psi_{I_i}(t) = c_{i-1}$
and $\lim_{t \to 1} \psi_{I_i}(t) = c_i$. Let
$\mathcal I_k$ be the set of connected components of $\mathbb
T \backslash \phi^{-k}(\mathcal D)$. Since $\mathcal C_k \subseteq \phi^{-k}(\mathcal D)$, the map $x
\mapsto \val \left(\phi^k,x\right)$ is constant on each $I \in
\mathcal I_k$ and we set $\val\left(\phi^k, I\right) =
\val\left(\phi^k,x\right), x \in I$. 
Set
$$
\mathcal I_k^{(2)} = \left\{ (I,J) \in \mathcal I_k
  \times \mathcal I_k : \ \phi^k(I) = \phi^k(J), \ \val
  \left(\phi^k ,I\right) = \val
    \left(\phi^k, J\right) \right\} .
$$
Let $\mathbb B_k$ denote
the finite-dimensional $C^*$-algebra generated by the matrix units
$e_{I,J}$, where $(I,J) \in \mathcal I_k^{(2)}$. Similarly, we let
$\mathbb A_k$ be the
  finite-dimensional $C^*$-algebra generated by the matrix units
  $e_{x,y}$ where $x,y \in \phi^{-k}(\mathcal D)$, $\phi^k(x) = \phi^k(y)$ and
  $\val\left(\phi^k,x\right) = \val\left(\phi^k, y\right)$.

When $x \in \phi^{-k}(\mathcal D)$
and $I \in \mathcal I_k$, write $I > x$ when $x \in
\overline{I}$ and $y > x$ for all $y \in I$, and $I < x$ when $x \in
\overline{I}$ and $y < x$ for all $y \in I$. Define a
  $*$-homomorphism $I_k : \mathbb A_k \to \mathbb B_k$
  such that
$$
I_k\left(e_{x,y}\right) = \sum_{I,J} e_{I,J}
$$
where we sum over the set of pairs $(I,J) \in  \mathcal I_k^{(2)}$
with the properties that $x < I, y < J$ and $\val\left(\phi^k, I\right)
= (+,+)$, or $I < x, J<y$ and $\val \left( \phi^k,I\right) = (-,-)$.


Similarly, we define a $*$-homomorphism
$U_k : \mathbb A_k \to \mathbb B_k$ such that
$$
U_k\left(e_{x,y}\right) = \sum_{I,J} e_{I,J}
$$
where we sum over the set of pairs $(I,J) \in  \mathcal I_k^{(2)}$
with the properties that $I < x, J < y$ and
$\val\left(\phi^k,I\right) =(+,+)$, or $x <I, y < J$ and
$\val\left(\phi^k, I\right) = (-,-)$. 
Let $(I,J) \in \mathcal I_k^{(2)}$ such that $\phi^k(I) =
\phi^k(J) = I_i$. Let $\lambda_I : {I_i} \to {I}$ be the inverse of
$\phi^k : {I} \to {I_i}$, and similarly $\lambda_J : {I_i} \to {J}$ the inverse of
$\phi^k : {J} \to {I_i}$. Then
$$
\left(\lambda_I \circ \psi_{I_i}(t), 0, \lambda_J \circ \psi_{I_i}(t)\right) \in
R^+_{\phi}(k)
$$
for all $t \in ]0,1[$. Notice that the limits 
$$\overline{\lambda_I}(c_i)
= \lim_{x \to c_i} \lambda_I(x) \ \text{and} \ \overline{\lambda_I}\left(c_{i-1}\right) = \lim_{x \to c_{i-1}}
\lambda_I(x)
$$ 
both exist. Let $f \in C_c\left(R^+_{\phi}(k)\right)$. Then the function
$$
]0,1[ \ni t \mapsto f\left(\lambda_I \circ \psi_{I_i}(t), 0,\lambda_J \circ
  \psi_{I_i}(t)\right)
$$
has a unique continuous extension $f_{I,J} : [0,1] \to \mathbb
C$. This is because
$$
\lim_{t \to 1} f\left(\lambda_I \circ \psi_{I_i}(t), 0,\lambda_J \circ
  \psi_{I_i}(t)\right) = 0
$$
when $\left(\overline{\lambda_I}(c_i), 0,\overline{\lambda_J}(c_i) \right) \notin R^+_{\phi}(k)$
and
$$ 
\lim_{t \to 1} f\left(\lambda_I \circ \psi_{I_i}(t),0, \lambda_J \circ
  \psi_{I_i}(t)\right) = f\left(\overline{\lambda_I}(c_i),0,\overline{\lambda_J}(c_i) \right)
$$
when $\left(\overline{\lambda_I}(c_i), 0,\overline{\lambda_J}(c_i)
\right) \in R^+_{\phi}(k)$. Similarly,
$$
\lim_{t \to 0} f\left(\lambda_I \circ \psi_{I_i}(t), 0,\lambda_J \circ
  \psi_{I_i}(t)\right) = 0
$$
when $\left(\overline{\lambda_I}(c_{i-1}), 0,\overline{\lambda_J}(c_{i-1})
\right) \notin R^+_{\phi}(k)$,
and
$$ 
\lim_{t \to 0} f\left(\lambda_I \circ \psi_{I_i}(t), 0, \lambda_J \circ
  \psi_{I_i}(t)\right) = f\left(\overline{\lambda_I}(c_{i-1}),0,\overline{\lambda_J}(c_{i-1}) \right)
$$
when $\left(\overline{\lambda_I}(c_{i-1}), 0,\overline{\lambda_J}(c_{i-1})
\right) \in R^+_{\phi}(k)$.

We can then define a $*$-homomorphism $b :
C_c\left(R^+_{\phi}(k)\right) \to C\left([0,1], \mathbb
  B_k\right)$ such that
$$
b(f) = \sum_{I,J} f_{I,J}e_{I,J} . 
$$


We can also define a $*$-homomorphism $a :
C_c\left(R^+_{\phi}(k)\right) \to \mathbb A_k$ such that
$$
a(f) = \sum_{(x,y) \in \mathcal A_k} f(x,0,y)e_{x,y} ,
$$
where $\mathcal A_k = \left\{ (x,y) \in \mathbb T^2 : \  (x,0,y) \in
  R^+_{\phi}(k) \right\}$.
By construction
$I_k\left(a(f)\right) = b(f)(0)$ and $U_k(a(f)) = b(f)(1)$.

For $x
\in \mathbb T$, let $\pi_x$ be the $*$-representation used to define
the norm on $C^*_r\left(R^+_{\phi}(k)\right)$, cf. (\ref{pix}). When $x
\notin \phi^{-k}(\mathcal D)$ there is an $i$ and a $t$ such that
$\phi^k(x) = \psi_{I_i}(t)$. Then
\begin{equation}\label{norm1} 
\left\|\pi_x(f)\right\| =
\left\| \sum_{(I,J) \in B} f_{I,J}(t) e_{I,J}\right\|,
\end{equation}
where $B = \left\{(I,J) \in \mathcal I_k^{(2)} : \ \phi^k(I) =
  \phi^k(J) = I_i \right\}$. When $x \in
\phi^{-k}(\mathcal D)$, we find that 
\begin{equation}\label{norm2}
\left\|\pi_x(f)\right\| = \left\|\sum_{(z,y) \in A}
  f(z,0,y) e_{z,y} \right\|,
\end{equation}
where $A = \left\{(z,y) \in \mathbb T^2 : \ (z,0,y) \in
  R^+_{\phi}(k), \ \phi^k(z)= \phi^k(x) , \ \val(\phi^k,z) =
  \val(\phi^k,x) \right\}$. By combining (\ref{norm1}) and (\ref{norm2})
we find that $f \to
(a(f),b(f))$ is isometric and extends to an injective $*$-homomorphism 
\begin{equation*}\label{nhomiso}
\mu_k : C^*_r\left(R^+_{\phi}(k)\right) \to \left\{ (a,b) \in \mathbb A_k \oplus C\left([0,1], \mathbb B_k \right) : \
  I_k(a) = b(0), \ U_k(a) = b(1) \right\} .
\end{equation*}

\begin{lemma}\label{iso1II} $\mu_k $ is an isomorphism.
\end{lemma}
\begin{proof} It remains to show that $\mu_k$ is surjective. Let $(a,b)
  \in \mathbb A_k \oplus C\left([0,1], \mathbb B_k \right)$ have
  the properties that $
  I_k(a) = b(0), \ U_k(a) = b(1)$. Then 
$a = \sum_{(x,y) \in A}\lambda_{x,y} e_{x,y}$, where 
$$
A = \left\{ (x,y) \in \phi^{-k}(\mathcal D) \times \phi^{-k}(\mathcal
  D) : \ (x,0,y) \in R^+_{\phi}(k) \right\}
$$ 
and $\lambda_{x,y} \in \mathbb C$. Since $A$ is a finite set there is a function $g \in
  C_c\left(R^+_{\phi}(k)\right)$ such that $g(x,0,y) = \lambda_{x,y}$ for
  all $(x,y) \in A$. Then $(a,b) = \mu_k(g) + (0,b')$, where $b' \in
  C_0(]0,1[) \otimes \mathbb B_k$. Write $b' = \sum_{(I,J)\in
    \mathcal I_k^2} b'_{I,J} e_{I,J}$, where $b'_{I,J} \in
  C_0(]0,1[)$. For each $(I,J)$ in the last sum there is an interval
  $I_i$ such that $\phi^k(I) = \phi^k(J) = I_i$. Then 
$$
\mathcal U = \left\{\left(\lambda_I\circ \psi_{I_i}(t), 0, \lambda_J \circ
    \psi_{I_i}(t)\right) : \ t \in ]0,1[\right\} 
$$
is an open subset of $R^+_{\phi}(k)$ and we can define a function
$h_{I,J}$ on $\mathcal U$ such that
$$
h_{I,J} \left(\lambda_I\circ \psi_{I_i}(t), 0, \lambda_J \circ
    \psi_{I_i}(t)\right) = b'_{I,J}(t) .
$$
Then $h_{I,J} \in C_0(\mathcal U)$ and we can choose a sequence $\{h_n\}
  \subseteq C_c(\mathcal U)$ such that $\lim_{n \to \infty} h_n = h_{I,J}$,
  uniformly on $R^+_{\phi}(k)$. Since $\mathcal U$ is a bi-section in
  $R^+_{\phi}$ it follows that $\lim_{n \to \infty}
  \mu_k(h_n) = \left(0,b'_{I,J}e_{I,J}\right)$, cf. Lemma 2.4 in
  \cite{Th1}, proving that
  $b'_{I,J}e_{I,J}$ is in the range of $\mu_k$ for all $(I,J)$. It
  follows that $(a,b)$ is in the range of $\mu_k$.
\end{proof}

It follows from Lemma \ref{iso1II} that $C^*_r\left(R^+_{\phi}(k)\right)$
is a recursive sub-homogeneous $C^*$-algebra in
the sense of N.C. Phillips. In combination
with (\ref{union}) it follows that $C^*_r\left(R^+_{\phi}\right)$ is
an ASH-algebra with no dimension growth, cf. \cite{T}.


\begin{cor}\label{nucUCT2} The $C^*$-algebra
  $C^*_r\left(\Gamma^+_{\phi}\right)$ is nuclear and satisfies the universal coefficient theorem (UCT) of Rosenberg
  and Schochet, \cite{RS}.
\end{cor} 
\begin{proof} Since $C^*_r\left(R^+_{\phi}\right)$ is the inductive limit of a sequence
  of sub-homogeneous $C^*$-algebras, it is nuclear and satisfies
  the UCT. The corollary follows therefore from Theorem \ref{cunPi}
  and \cite{Ka} when $C^*_r\left(\Gamma^+_{\phi}\right)$ is simple. To obtain the same conclusion in general we must use a different and even more indirect path. First observe that the nuclearity of $C^*_r\left(R^+_{\phi}\right)$ implies that $R^+_{\phi}$ is (topologically) amenable by Corollary 6.2.14 (ii) and Theorem 3.3.7 in \cite{A-DR}, and then from Proposition 6.1.8 in \cite{A-DR} that $C^*_r\left(R^+_{\phi}\right)$ equals the full groupoid $C^*$-algebra of $R^+_{\phi}$. Since $C^*_r\left(R^+_{\phi}\right)$ is the fixed-point algebra of the gauge action it follows then from a recent result of Spielberg, Proposition 9.3 in \cite{Sp}, that $\Gamma^+_{\phi}$ is also (topologically) amenable. Then Corollary 6.2.14 (i) from \cite{A-DR} shows that $C^*_r\left(\Gamma^+_{\phi}\right)$ is nuclear, and the work of Tu in \cite{Tu} (more precisely, Lemme 3.5 and Proposition 10.7 in \cite{Tu}) implies that it satisfies the UCT. 
\end{proof}

We do not consider the full groupoid $C^*$-algebra here,
  \cite{Re}, but observe in passing that the previous proof established the topological amenability of $\Gamma^+_{\phi}$; a fact which implies that
  the full and
  reduced $C^*$-algebras of $\Gamma^+_{\phi}$ are canonically
isomorphic by Proposition 6.1.8 in \cite{A-DR}.

\section{Markov maps}\label{MarMaps}

We say that
$\phi$ is \emph{Markov} when $\phi(\mathcal C_1) \subseteq \mathcal
C_1$. 

\begin{lemma}\label{simpmarko} Assume that $\phi$ is Markov. Then
  $C^*_r\left(\Gamma^+_{\phi}\right)$ is simple if and only if $\phi$
  is transitive.
\end{lemma}
\begin{proof} The Markov condition implies that all post-critical points
are critical and it follows therefore from Lemma \ref{4a} that a transitive Markov map has no
exposed points and therefore also that
$C^*_r\left(\Gamma^+_{\phi}\right)$ is simple when $\phi$ is
transitive and Markov.
\end{proof}

\begin{lemma}\label{mar1X} Assume that $\phi$ is Markov and
  transitive. There is a natural number $k \in \mathbb N$ with the property
  that when $j \geq k$ and
  $x \in \phi\left(\mathbb T \backslash \mathcal C_1\right)$, there are elements
  $y_{\pm} \in \mathbb T$ such that $\phi^j\left(y_{\pm}\right) = x$,
  $\val \left(\phi^j,y_+\right) = (+,+)$ and $\val
  \left(\phi^j,y_-\right) = (-,-)$.
\end{lemma}
\begin{proof} Since $\phi$ is not locally injective there are
  non-empty open intervals $I_{\pm}$ such that $\val\left(\phi,x\right) =
    (+,+)$ for all $x \in I_+$, $\val \left(\phi,y \right) = (-,-)$
    for all $y \in I_-$ and $I = \phi(I_+)= \phi(I_-)$ is an open
    non-empty interval. Since $\phi$ is exact (by Lemma
    \ref{simpmarko} and Theorem \ref{S}) there is an $ N \in
    \mathbb N$ such that $\phi^N(I) = \mathbb T$. Set $k =N + 2$ and
    let $j \geq k$. Consider an element $x \in \phi\left(\mathbb T
      \backslash \mathcal C_1\right)$. Then $x = \phi(u)$ for some $u
    \in \mathbb T  \backslash \mathcal C_1$. There is an element $z \in I$
    such that $\phi^{j-2}(z) = u$ and elements $z_{\pm} \in I_{\pm}$ such
    that $z_{\pm} \in I_{\pm}$ and $\phi(z_{\pm}) = z$. Then
    $\phi^j(z_{\pm}) = x$ and the Markov condition implies that
    $z_{\pm}$ are not critical for $\phi^{j-1}$ since $u \notin
    \mathcal C_1$. Note that $\val\left(\phi^j, z_{\pm}\right) =
    \val\left(\phi^{j-1},z\right)
    \bullet \val\left( \phi,  z_{\pm}\right) $. If $\val \left(\phi^{j-1},z\right) = (+,+)$,
    set $y_+ = z_+$ and $y_- = z_-$, and if $\val
    \left(\phi^{j-1},z\right) = (-,-)$, set $y_+ = z_-$ and $y_- = z_+$.  
\end{proof}

In the following we say that a Markov map $\phi$ is of \emph{order
  $k$} when the conclusion of Lemma \ref{mar1X} holds; i.e. when
\begin{enumerate}\label{k}
\item[a)] for all $j \geq k$ and all $x \in \phi(\mathbb T \backslash
  \mathcal C_1)$ there are elements
  $y_{\pm} \in \mathbb T$ such that $\phi^j\left(y_{\pm}\right) = x$,
  $\val \left(\phi^j,y_+\right) = (+,+)$ and $\val
  \left(\phi^j,y_-\right) = (-,-)$.
\end{enumerate}

Besides our standing assumptions (that $\phi : \mathbb T \to \mathbb
T$ is continuous, piecewise monotone and not locally injective), we now also
assume that $\phi$ is a Markov map of order $k$.

\begin{lemma}\label{nyk} Assume that $(x,l,y) \in \Gamma^+_{\phi}(l,n)$, where
  $l \in \mathbb Z, \ n \in \mathbb N$ and $l+n \geq k+1$. It follows that there is an element $z \in \mathbb T$
  such that $(x,1,z) \in \Gamma^+_{\phi}(1,k)$ and $(z,l-1,y) \in
  \Gamma^+_{\phi}(l -1,n)$.
\end{lemma}
\begin{proof} If $\phi^k(x) \in \mathcal C_1$, set $z
  = \phi(x)$. It follows from Lemma \ref{compx} and the composition
  table for $\bullet$ that
\begin{equation*}
\begin{split}
& \val\left( \phi^{k+1}, x\right) = \val\left( \phi, \phi^k(x)\right)
\bullet \val\left( \phi^k,x\right) = \\
&\val\left( \phi, \phi^k(x)\right)  \bullet \val \left(\phi^{k-1},
  \phi(x)\right) = \val\left(\phi^k,\phi(x)\right),
\end{split}
\end{equation*}
showing that $(x,1,z) \in \Gamma^+_{\phi}(1,k)$. Now note that it
follows from Lemma \ref{compx} and the composition
  table for $\bullet$ first that $\phi^k(x)$ is critical for
  $\phi^{n+l-k}$ and then, as above, that
\begin{equation*}
\begin{split}
& \val\left( \phi^{n}, y\right) = \val\left( \phi^{n+l}, x\right) = \val\left( \phi^{n+l-k}, \phi^k(x)\right)
\bullet \val\left( \phi^k,x\right) = \\
&\val\left( \phi^{n+l-k}, \phi^k(x)\right)  \bullet \val \left(\phi^{k-1},
  z\right) = \val\left(\phi^{n+l-1},z\right).
\end{split}
\end{equation*}
This shows that
$(z,l-1,y) \in
  \Gamma^+_{\phi}(l -1,n)$.

If instead $\phi^k(x) \notin \mathcal C_1$ it follows from condition
a) that there is a $z \in
\mathbb T$ such that $\phi^{k+1}(x) = \phi^k(z)$ and
$\val\left(\phi^k,z\right) = \val\left(\phi^{k+1},x\right)$. Then $(x,1,z)
\in \Gamma^+_{\phi}(1,k)$ and
$(z,l-1,y) \in \Gamma^+_{\phi}(l-1,n)$ since 
\begin{equation*}
\begin{split}
& \val \left(
  \phi^{l-1+n},z\right)  = \val \left(
  \phi^{l+n-k-1},\phi^k(z)\right) \bullet \val\left(\phi^k,z\right)\\
&=  \val \left(
  \phi^{l+n-k-1},\phi^{k+1}(x)\right) \bullet
\val\left(\phi^{k+1},x\right) 
= 
\val\left(\phi^{l+n},x\right) = 
\val\left(\phi^{n},y\right)  
\end{split}
\end{equation*}
when $l+ n \geq k+2$, while 
\begin{equation*}
\begin{split}
& \val \left(
  \phi^{l-1+n},z\right)  = \val \left( \phi^k,z\right) \\
&=  \val \left(
  \phi^{k+1},x\right) = \val\left(\phi^{l+n}, x\right) = 
\val\left(\phi^{n},y\right)  
\end{split}
\end{equation*}
when $l+n = k+1$.
\end{proof}

\begin{lemma}\label{mar7k} Assume that $n \geq k+1$. Let $(x,0,y) \in
  \Gamma^+_{\phi}(0,n)$. There are elements $z_1,z_2 \in \mathbb T$
  such that $(z_1,0,z_2) \in \Gamma^+_{\phi}(0,n-1)$, $(x,1,z_1),
  (y,1,z_2) \in \Gamma^+_{\phi}(1,k)$ and
$$
(x,0,y) = (x,1,z_1)(z_1,0,z_2)(z_2,-1,y)
$$
in $\Gamma^+_{\phi}$.
\end{lemma}
\begin{proof} We obtain from Lemma \ref{nyk} an element $z_1 \in
  \mathbb T$ such that $(x,1,z_1) \in \Gamma^+_{\phi}(1,k)$,
  $(z_1,-1,y) \in \Gamma^+_{\phi}(-1,n)$ and $(x,0,y) = (x,1,z_1)(z_1,-1,y)$. Then $(y,1,z_1) \in
  \Gamma^+_{\phi}(1,n-1)$ and a second application of Lemma
  \ref{nyk} gives a $z_2 \in
  \mathbb T$ such that $(y,1,z_2) \in \Gamma^+_{\phi}(1,k)$,
  $(z_2,0,z_1) \in \Gamma^+_{\phi}(0,n-1)$ and $(y,1,z_1) = (y,1,z_2)(z_2,0,z_1)$.
\end{proof}

 Let $E$ be the closure of
  $C_c\left(\Gamma^+_{\phi}(1,k)\right)$ in
  $C^*_r\left(\Gamma^+_{\phi}\right)$. Then
$$
E^*E \subseteq C^*_r\left(R^+_{\phi}(k)\right), \
C^*_r\left(R^+_{\phi}(k+1)\right)E \subseteq E, \ E
C^*_r\left(R^+_{\phi}(k)\right)\subseteq E.
$$
We can therefore, in the natural way, consider $E$ as a $C^*$-correspondence on $C^*_r\left(R^+_{\phi}(k)\right)$ in the
  sense of Katsura, \cite{Ka}. Let $\mathcal O_E$ denote the
  corresponding $C^*$-algebra, cf. Definition 3.5 in \cite{Ka}. We aim
  now to show that $\mathcal O_E$ is isomorphic to
  $C^*_r\left(\Gamma^+_{\phi}\right)$.

\begin{lemma}\label{algcon2k}
\begin{enumerate}
\item[1)] $\overline{\Span EE^*} = C^*_r\left(R^+_{\phi}(k+1)\right)$.
\item[2)] $C^*_r\left(R^+_{\phi}(n)\right)  \subseteq
  EC^*_r\left(R^+_{\phi}(n-1)\right)E^*$ when $n \geq k+1$.
\item[3)] $C^*_r\left(\Gamma^+_{\phi}\right)$ is generated by $E$.
\end{enumerate}
\end{lemma}
\begin{proof}
1) It is clear that $EE^* \subseteq
C^*_r\left(R^+_{\phi}(k+1)\right)$. Let $h \in
C_c\left(R^+_{\phi}(k+1)\right)$. To show that $h$ is contained in
$\overline{\Span E E^*}$
we may assume that $h$ is supported in a bi-section $U$. Set $K =
r\left(\supp h\right) \subseteq \mathbb T$. Let $x \in K$. There is
then a unique $y \in \mathbb T$ such that $(x,0,y) \in
U$. It follows from Lemma \ref{mar7k} that there
is a $z \in \mathbb T$ such that $(x,1,z), (y,1,z) \in
\Gamma^+_{\phi}(1,k)$. Choose open bi-sections $W_1 \subseteq \Gamma^+_{\phi}(1,k), \ W_2 \subseteq \Gamma^+_{\phi}(-1,k+1)$ containing
$(x,1,z)$ and $(z,-1,y)$, respectively, such that $W_1W_2 \subseteq U$. We can then define functions $f \in C_c(W_1), \ g \in C_c(W_2)$ such that $f(\gamma) = h(\gamma')$ in a neighborhood of $(x,1,z)$, where $\gamma' \in U$ is determined by the condition that $r(\gamma) = r(\gamma')$, and $g(\gamma) = 1$ in a neighborhood of $(z,-1,y)$. There is then an open neighborhood $V_x$ of $x$ such that
$fg(\gamma) = h(\gamma)$ for all $\gamma \in r^{-1}(V_x)$. By compactness of $K$ we get in this way a finite collection of functions $\psi_i \in C(\mathbb T), f_i ,g_i \in C_c\left(\Gamma^+_{\phi}(1,k)\right)$ such that $h = \sum_i \psi_if_ig_i$. Since $\psi_if_i, g_i^* \in E$ this shows that $h$ is in the span of $EE^*$.

2) follows from Lemma \ref{mar7k} in a similar way.

3) By arguments similar to those used for 1) it follows from Lemma
\ref{nyk} that $C_c\left(\Gamma^+_{\phi}(l)\right) \subseteq \Span EC_c\left(\Gamma^+_{\phi}(l-1)\right)$ for all $l \in \mathbb
Z$. Since $C_c\left(\Gamma^+_{\phi}(l)\right)^* =
C_c\left(\Gamma^+_{\phi}(-l)\right)$ and since $\bigcup_{l \in \mathbb
  Z} C_c\left(\Gamma^+_{\phi}(l)\right)$ is dense in
$C^*_r\left(\Gamma^+_{\phi}\right)$ it suffices then to show that
$C^*_r\left(R^+_{\phi}\right)$ is contained in the $C^*$-algebra
generated by $E$. This follows from 1), 2) and (\ref{union}).
\end{proof}

Let $\mathbb L(E)$ denote the $C^*$-algebra of adjointable operators
on the Hilbert $C^*_r\left(R^+_{\phi}(k)\right)$-module $E$ and 
$\mathbb K(E)$ the ideal in $\mathbb L(E)$ consisting of the 'compact' operators.

\begin{lemma}\label{mar4k}
Define $\pi : C^*_r\left(R^+_{\phi}(k+1)\right) \to \mathbb L(E)$ such
that $\pi(a)b= ab$. Then $\pi$ is injective and
$\pi\left(C^*_r\left(R^+_{\phi}(k+1)\right)\right) = \mathbb K(E)$.
\end{lemma}
\begin{proof} This follows from 1) in Lemma \ref{algcon2k}.
\end{proof}

\begin{thm}\label{katsura2k} Assume that $\phi$ is Markov of order $k$. Then
\begin{equation}\label{isoIIk}
C^*_r\left(\Gamma^+_{\phi}\right) \simeq \mathcal O_E,
\end{equation}
and there is a six-terms exact sequence
\begin{equation}\label{6termIIk}
\begin{xymatrix}{
K_0\left( C^*_r\left(R^+_{\phi}(k)\right)\right) \ar[rr]^-{\id -
  [E]_0} & & K_0\left(
    C^*_r\left(R^+_{\phi}(k)\right)\right) \ar[rr]^{\iota_0}  &
& K_0\left( C^*_r\left(\Gamma^+_{\phi}\right) \right)\ar[d] \\
K_1\left(C^*_r\left(\Gamma^+_{\phi}\right) \right)\ar[u]  & &
K_1\left(C^*_r\left(R^+_{\phi}(k)\right) \right) \ar[ll]^{\iota_1}  & &  K_1\left( C^*_r\left(R^+_{\phi}(k)\right)\right) \ar[ll]^-{\id -
 [E]_1}
}\end{xymatrix}
\end{equation} 
where $\iota :
C^*_r\left(R^+_{\phi}(k)\right) \to
C^*_r\left(\Gamma^+_{\phi}\right)$ is the inclusion map and $[E]$ is the KK-theory element defined from $\pi :
C^*_r\left(R^+_{\phi}(k)\right) \to \mathbb K(E)$.
\end{thm}
\begin{proof} Let $t : E \to  C^*_r\left(\Gamma^+_{\phi}\right)$ and
  $\iota : C^*_r\left(R^+_{\phi}(k)\right) \to
  C^*_r\left(\Gamma^+_{\phi}\right)$ be
  the embeddings. It follows then from Lemma \ref{mar4k} above and
  Proposition 3.3 in \cite{Ka} that $(\iota, t)$ is covariant in the
  sense of Definition 3.4 in \cite{Ka}. The isomorphism (\ref{isoIIk})
  follows then from 3) in Lemma \ref{algcon2k} and Theorem 6.4 in
  \cite{Ka}. Then the 6-terms exact sequence (\ref{6termIIk}) follows
 from Lemma \ref{mar4k} above and Theorem 8.6 in \cite{Ka}.  
\end{proof}

\subsection{The linking groupoid}

\begin{lemma}\label{morita3k} Let $(x,0,y) \in R^+_{\phi}(k)$. There is
  an element $z \in \mathbb T$ such that 
$$
(z,1,x),(z,1,y) \in
  \Gamma^+_{\phi}(1,k)
$$ 
and $(x,0,y) = (x,-1,z)(z,1,y)$ in $\Gamma^+_{\phi}$.
\end{lemma}
\begin{proof} If $\phi^{k-1}(x) \in \mathcal C_1$, choose any $z \in
  \phi^{-k}\left( \phi^{k-1}(x)\right)$ and note that $\val \left(
    \phi^{k+1},z\right) = \val\left( \phi, \phi^{k-1}(x)\right) =
  \val\left(\phi^k , x\right)$. It follows that $(z,1,x), (z,1,y) \in
  \Gamma^+_{\phi}(1,k)$.

If $\phi^{k-1}(x)  \notin \mathcal C_1$ it follows from the Markov
condition and the surjectivity of $\phi$ that $\phi^{k-1}(x) \in
\phi\left(\mathbb T \backslash \mathcal C_1\right)$. Furthermore, the
Markov condition also ensures that $\val\left( \phi^{k-1},x\right) \in
(\pm,\pm)$. Since $\phi$ is
Markov of order $k$ there is therefore an element $z \in \mathbb T$ such that
$\val\left(\phi^k,z\right) = \val \left(\phi^{k-1},x\right)$ and
$\phi^k(z) = \phi^{k-1}(x)$. It
follows that $(z,1,x), (z,1,y) \in \Gamma^+_{\phi}(1,k)$. 
 \end{proof}

It follows from Lemma \ref{nyk} and Lemma \ref{morita3k} that $\Gamma^+_{\phi}(1,k)$ is a $R^+_{\phi}(k+1)-R^+_{\phi}(k)$-equivalence in the sense of \cite{SW} and hence from Theorem 13 in \cite{SW} that $C^*_r\left(R^+_{\phi}(k)\right)$ is Morita equivalent to $C^*_r\left(R^+_{\phi}(k+1)\right)$. We need a detailed description of the corresponding linking algebra. To this end we introduce the the linking groupoid $\mathcal L$ as follows, cf. Lemma 3 of
\cite {SW}. As a topological space $\mathcal L$ is the disjoint union
$$
\mathcal L = R^+_{\phi}(k+1) \sqcup R^+_{\phi}(k) \sqcup
\Gamma^+_{\phi}(1,k) \sqcup \Gamma^+_{\phi}(-1,k+1) .
$$
To give $\mathcal L$ a groupoid structure, define $r,s : \mathcal L \to
\mathbb T \sqcup \mathbb T \subseteq R^+_{\phi}(k+1) \sqcup
R^+_{\phi}(k) \subseteq \mathcal L$ to be the maps coming from the
range and source maps on $\Gamma^+_{\phi}$, but such that $r$ takes
$R^+_{\phi}(k+1) \sqcup \Gamma^+_{\phi}(1,k)$ to the first copy of
$\mathbb T$, and $R^+_{\phi}(k) \sqcup \Gamma^+_{\phi}(-1,k+1)$ to the
second while $s$ takes $R^+_{\phi}(k+1) \sqcup \Gamma^+_{\phi}(-1,k+1)$ to the first copy of
$\mathbb T$, and $R^+_{\phi}(k) \sqcup \Gamma^+_{\phi}(1,k)$ to the
second. We then set $\mathcal L^{(2)} = \left\{ (\gamma, \gamma') \in
  \mathcal L \times \mathcal L: \ s(\gamma) = r(\gamma') \right\}$ and
define the product $\gamma \gamma'$ for $(\gamma, \gamma') \in
\mathcal L^{(2)}$ to be the same as the product in
$\Gamma^+_{\phi}$. It is then straightforward to verify that $\mathcal
L$ is a second countable \'etale locally compact Hausdorff
groupoid. 

The $C^*$-algebra of the reduction of $\mathcal L$ to the first copy of $\mathbb T$ is in a natural way identified with $C^*_r\left(R^+_{\phi}(k+1)\right)$
and in the same way the $C^*$-algebra of the reduction of $\mathcal L$ to the second copy of
$\mathbb T$ is identified with
$C^*_r\left(R^+_{\phi}(k)\right)$. These identifications give us
embeddings $a : C^*_r\left(R^+_{\phi}(k+1)\right) \to C_r^*(\mathcal
L)$ and $b : C^*_r\left(R^+_{\phi}(k)\right) \to C_r^*(\mathcal
L)$ onto full corners of $C^*_r(\mathcal L)$. To see how $a$ and $b$
are related to $[E]$, let $f \in C_c(\mathcal L)$ and write $f =
f_{11} + f_{12} + f_{21} + f_{22}$ where $f_{11} \in
C_c\left(R^+_{\phi}(k+1)\right)$, $f_{12} \in
C_c\left(\Gamma^+_{\phi}(1,k)\right)$, $f_{21} \in
C_c\left(\Gamma^+_{\phi}(-1,k+1)\right)$ and $f_{22} \in
C_c\left(R^+_{\phi}(k)\right)$. We can then define $\Psi(f) \in
\mathbb L\left(E \oplus R^+_{\phi}(k)\right)$ such that
$$
\Psi(f)(e,g) = \left(f_{11}e + f_{12}g, f_{21}e + f_{22} g \right) .
$$ 
It follows from Theorem 13 in \cite{SW} and Corollary 3.21 in \cite{RW} that $\Psi$ extends to a $*$-isomorphism 
$$
C^*_r(\mathcal L) \to
\mathbb K\left(E \oplus C^*_r\left(R^+_{\phi}(k)\right)\right).
$$ 
Let $
\mathbb K(E) \to \mathbb K\left(E \oplus
  C^*_r\left(R^+_{\phi}(k)\right)\right)$ and $ 
C^*_r\left(R^+_{\phi}(k)\right) \to \mathbb K\left(E \oplus
  C^*_r\left(R^+_{\phi}(k)\right)\right)$ be the canonical
embeddings. Then

\begin{equation*}
\begin{xymatrix}{
C^*_r\left(R^+_{\phi}(k+1)\right) \ar[r]^-a \ar[d]^-{\pi} &
C^*_r(\mathcal L) \ar[d]^-{\Psi} &
C^*_r\left(R^+_{\phi}(k)\right) \ar[l]_-b \ar@{=}[d] \\
\mathbb K(E) \ar[r] & \mathbb K\left(E \oplus
  C^*_r\left(R^+_{\phi}(k)\right)\right) &
C^*_r\left(R^+_{\phi}(k)\right) \ar[l] 
}
\end{xymatrix}
\end{equation*} 
commutes, when $\pi$ is the isomorphism from Lemma \ref{mar4k}. By
relating this diagram to the description of $[E]_*$ given by
Definition 8.3 and Appendix A in \cite{Ka} we find that

\begin{equation}\label{comp}
[E]_* = b_*^{-1} \circ a_* \circ \rho_* . 
\end{equation}
To study $[E]_*$ further we need information on the structure of $C^*_r\left(\mathcal L\right)$ and hence the following observation.

\begin{lemma}\label{kstorjIIk} $\phi^j\left(\mathcal C_j\right) =
  \phi^{j+1}\left(\mathcal C_{j+1}\right)$ for all $j \geq 1$.
\end{lemma}
\begin{proof}

The inclusion $\phi^j\left(\mathcal C_j\right) \subseteq
\phi^{j+1}\left(\mathcal C_{j+1}\right)$ is a general fact and follows
from the observation that $\phi^{-1}\left(\mathcal C_j\right)
\subseteq \mathcal C_{j+1}$. For the converse, note that $\mathcal C_{j+1} = \mathcal C_1 \cup
\phi^{-1}(\mathcal C_1) \cup \cdots \cup \phi^{-j}\left(\mathcal
  C_1\right)$ and by the Markov property $\phi\left(C_{j+1}\right) \subseteq \phi(\mathcal C_1) \cup
\phi\left(\phi^{-1}(\mathcal C_1)\right) \cup \cdots \cup \phi\left(\phi^{-j}\left(\mathcal
  C_1\right)\right) \subseteq \mathcal C_1 \cup \mathcal C_1 \cup
\phi^{-1}(\mathcal C_1) \cup \cdots \cup \phi^{-j+1}(\mathcal C_1) =
\mathcal C_j$,
which implies the desired inclusion. 
\end{proof}

\begin{cor}\label{corTTTIIk} $\phi^j(\mathcal C_j) = \phi(\mathcal C_1)$
  when $j \geq 1$.
\end{cor}

Set $\mathcal D =
\phi(\mathcal C_1)$. Then $\mathcal C_j \subseteq \phi^{-j}(\mathcal D)$
by Corollary \ref{corTTTIIk}. Note that $\phi^{-k-1}\left(\mathcal D\right) \sqcup
\phi^{-k}(\mathcal D)$ is an $\mathcal L$-invariant subset of $\mathbb
T \sqcup \mathbb T$. We get therefore an extension
\begin{equation}\label{ki1IIk}
\begin{xymatrix}{
0 \ar[r] & C^*_r \left(\mathcal L|_{(\mathbb T \sqcup \mathbb T)
    \backslash (\phi^{-k-1}(\mathcal D)\sqcup \phi^{-k}(\mathcal D))}\right) \ar[r] & C^*_r(\mathcal L)
\ar[r] &   C^*_r \left(\mathcal L|_{\phi^{-k-1}(\mathcal D)\sqcup \phi^{-k}(\mathcal D)}\right)
\ar[r] & 0
}\end{xymatrix}
\end{equation}

By using the same procedure as in Section \ref{extMar3II} we can give the
following alternative description of this extension. As in Section
\ref{extMar3II} let $\mathcal I_{k+1}$ be the set of intervals of connected
components in $\mathbb T \backslash \phi^{-k-1}(\mathcal D)$, considered as subsets of
the first copy of $\mathbb T$ in $\mathbb T  \sqcup \mathbb T$, and
$\mathcal I_k$ the set of intervals of connected components of
$\mathbb T \backslash \phi^{-k}( \mathcal D)$, considered as subsets of
the second copy of $\mathbb T$ in $\mathbb T  \sqcup \mathbb T$. Let $\mathbb B_{\mathcal L}$ be the finite-dimensional
$C^*$-algebra generated by the matrix units $e_{I,J}, I,J \in \mathcal
I_{k+1} \cup \mathcal I_{k}$, subject to the conditions that
\begin{enumerate}
\item[a)] $\phi^{k+1}(I) = \phi^{k+1}(J)$ and $\val
  \left(\phi^{k+1},I\right) = \val \left( \phi^{k+1},J\right)$ when
  $I,J \in \mathcal I_{k+1}$,
\item[b)] $\phi^{k}(I) = \phi^{k}(J)$ and $\val
  \left(\phi^{k},I\right) = \val \left( \phi^{k},J\right)$ when
  $I,J \in \mathcal I_{k}$,
\item[c)] $\phi^{k+1}(I) = \phi^{k}(J)$ and $\val
  \left(\phi^{k+1},I\right) = \val \left( \phi^{k},J\right)$ when $I
  \in \mathcal I_{k+1}$ and $J \in \mathcal I_k$,
\item[d)] $\phi^{k}(I) = \phi^{k+1}(J)$ and $\val
  \left(\phi^{k},I\right) = \val \left( \phi^{k+1},J\right)$ when
  $I \in \mathcal I_k$ and $J \in \mathcal I_{k+1}$.
\end{enumerate}
Then 
$$
C^*_r \left(\mathcal L|_{(\mathbb T \sqcup \mathbb T)
    \backslash (\phi^{-k-1}(\mathcal D)\sqcup \phi^{-k}(\mathcal
    D))}\right) \simeq S\mathbb B_{\mathcal L} ,
$$
where we have used the notation $SD$ for the suspension of a $C^*$-algebra $D$, i.e. $SD =
C_0(\mathbb R) \otimes D \simeq C_0(]0,1[) \otimes D$.

Similarly, we can describe the $C^*$-algebra $C^*_r \left(\mathcal
  L|_{\phi^{-k-1}(\mathcal D)\sqcup \phi^{-k}(\mathcal D)}\right)$ as the finite-dimensional
$C^*$-algebra $\mathbb A_{\mathcal L}$ generated by the matrix units $e_{x,y},
x,y \in \phi^{-k-1}(\mathcal D)\sqcup \phi^{-k}(\mathcal D) \subseteq \mathbb T \sqcup \mathbb T$, subject to the conditions that
\begin{enumerate}
\item[e)] $\phi^{k+1}(x) = \phi^{k+1}(y)$ and $\val
  \left(\phi^{k+1},x\right) = \val \left( \phi^{k+1},y\right)$ when
  $x,y \in \phi^{-k-1}(\mathcal D)$,
\item[f)] $\phi^k(x) = \phi^k(y)$ and $\val
  \left(\phi^k,x\right) = \val \left( \phi^k,y\right)$ when
  $x,y \in \phi^{-k}(\mathcal D)$,
\item[g)] $\phi^{k+1}(x) =\phi^k(y)$ and $\val
  \left(\phi^{k+1},x\right) = \val \left( \phi^k,y\right)$ when $x
  \in \phi^{-k-1}(\mathcal D)$ and $y \in \phi^{-k}(\mathcal D)$,
\item[h)] $\phi^k(x) = \phi^{k+1}(y)$ and $\val
  \left(\phi^k,x\right) = \val \left( \phi^{k+1},y\right)$ when
  $x \in \phi^{-k}(\mathcal D)$ and $y \in \phi^{-k-1}(\mathcal D)$. 
\end{enumerate}
Then 
$$
C^*_r \left(\mathcal
  L|_{\phi^{-k-1}(\mathcal D)\sqcup \phi^{-k}(\mathcal D)}\right)  \simeq \mathbb
A_{\mathcal L},
$$
and the extension (\ref{ki1IIk}) takes the form
\begin{equation}\label{ki2II}
\begin{xymatrix}{
0 \ar[r] &  S\mathbb B_{\mathcal L} \ar[r] & C^*_r(\mathcal L)
\ar[r] &   \mathbb
A_{\mathcal L}
\ar[r] & 0
}\end{xymatrix}
\end{equation}
This extension is compatible with the description of $C^*_r\left(R^+_{\phi}(k)\right)$ and $C^*_r\left(R^+_{\phi}(k+1)\right)$ coming from Section \ref{extMar3II} in the sense that we get a commutative diagram of $*$-homomorphisms
\begin{equation}\label{diag19II}
\begin{xymatrix}{
0 \ar[r] & S\mathbb B_{k+1} \ar[r] \ar[d] & C^*_r\left(
  R^+_{\phi}(k+1)\right) \ar[r] \ar[d]^-{a} & \mathbb A_{k+1} \ar[r] \ar[d]
& 0 \\
0 \ar[r] & S\mathbb B_{\mathcal L}
\ar[r] & C^*_r\left( \mathcal L\right) \ar[r] & \mathbb A_{\mathcal L} \ar[r] & 0 \\
0 \ar[r] &  S\mathbb B_k \ar[u] \ar[r]  & C^*_r\left(
  R^+_{\phi}(k)\right) \ar[u]_-{b} \ar[r]  & \mathbb A_k
\ar[u] \ar[r]
& 0 .
}\end{xymatrix}
\end{equation}
From the six-terms
exact sequences of the upper and lower extensions in (\ref{diag19II})
we get for $ j = k$ and $j = k+1$ the identifications
\begin{equation}\label{kerid}
K_0\left(C^*_r\left(  R^+_{\phi}(j)\right)\right) = \ker \left( \left(I_j\right)_0
      - \left(U_j\right)_0 \right) 
\end{equation}
and
\begin{equation}\label{cokerid}
K_1\left( C^*_r\left( R^+_{\phi}(j)\right) \right) = \coker \left( \left(I_j\right)_0
      - \left(U_j\right)_0 \right) ,
\end{equation}
where $I_j, U_j : \mathbb A_j \to \mathbb B_j$
are the $*$-homomorphisms introduced in Section \ref{extMar3II}. Then
$b_0^{-1} \circ a_0$ is realized as a homomorphism
$$
b_0^{-1} \circ a_0 : \ker \left( \left(I_{k+1}\right)_0
      - \left(U_{k+1}\right)_0 \right) \to \ker \left( \left(I_k\right)_0
      - \left(U_k\right)_0 \right) 
$$ 
and $b_1^{-1} \circ a_1$ as a
    homomorphism 
$$
b_1^{-1} \circ a_1 : \coker \left( \left(I_{k+1}\right)_0
      - \left(U_{k+1}\right)_0 \right) \to \coker \left( \left(I_k\right)_0
      - \left(U_k\right)_0 \right) .
$$

Note that the full matrix summands in each of the $C^*$-algebras $\mathbb B_k, \mathbb
B_{\mathcal L}$ and
$\mathbb B_{k+1}$ are in one-to-one correspondence with $\mathcal
I(\pm) = \mathcal I \times \{(+,+),(-,-)\}$, where $\mathcal I$ are
the intervals of connected components in $\mathbb T \backslash \mathcal D$. In $\mathbb
B_k$, for example, the element $(I, (+,+)) \in \mathcal I(\pm)$ labels the $C^*$-subalgebra of $\mathbb B_k$ generated by the matrix units $e_{I',J'}$ where $I',J'$ are intervals in
$\mathcal I_k$ such that $\phi^k(I') = \phi^k(J') = I$ and $\val\left(\phi^k,I'\right) = \val\left(\phi^k,J'\right) =( +,+)$, while the element $(I, (-,-)) \in \mathcal I(\pm)$ labels the $C^*$-subalgebra of $\mathbb B_k$ generated by the matrix units $e_{I',J'}$ where $I',J'$ are intervals in
$\mathcal I_k$ such that $\phi^k(I') = \phi^k(J') = I$ and $\val\left(\phi^k,I'\right) = \val\left(\phi^k,J'\right) =( -,-)$. Similarly, the full matrix
summands in $\mathbb A_k, \mathbb A_{\mathcal L}$ and $\mathbb A_{k+1}$ are in one-to-one
correspondence with the subset $\mathcal D(\pm)$ of $\mathcal D  \times
\mathcal V$ consisting of the pairs $(d,v) \in \mathcal D \times
\mathcal V$ for which there exists an element $x \in
\phi^{-k}(d)$ such that $\val\left(\phi^k,x\right) = v$. By using
these labels we get isomorphisms 
$$
K_0\left(\mathbb A_k\right) \simeq
K_0\left(\mathbb A_{k+1}\right) \simeq K_0\left(\mathbb A_{\mathcal
    L}\right) \simeq \mathbb Z^{\mathcal D(\pm)}
$$
and 
$$
K_0\left(\mathbb B_k\right) \simeq
K_0\left(\mathbb B_{k+1}\right) \simeq K_0\left( \mathbb B_{\mathcal
    L}\right) \simeq \mathbb Z^{\mathcal I(\pm)} \simeq \mathbb Z^{\mathcal I} \oplus
\mathbb Z^{\mathcal I}
$$ 
with the property that by making them
identifications the maps ${b_0}^{-1} \circ {a_0}$ and
${b_1}^{-1} \circ {a_1}$ become
identities. Therefore,
in order to obtain the desired description of $[E]_*$ we must
determine what the map $\rho_*
: K_*\left(C^*_r\left( R^+_{\phi}(k)\right)\right) \to
K_*\left(C^*_r\left( R^+_{\phi}(k+1)\right)\right)$ becomes under
these identifications.

\subsection{A description of $\rho : C^*_r\left(R^+_{\phi}(k)\right)
  \to  C^*_r\left(R^+_{\phi}(k+1)\right)$}

Using the notation from Section \ref{extMar3II}, set
$$
\mathbb D_j = \left\{ (a,f)  \in
  \mathbb A_j, f \in C\left([0,1],\mathbb B_j\right) : \ f(0) =
  I_j(a), \ f(1) = U_j(a) \right\}
$$
when $j \geq 1$. It follows from Lemma \ref{iso1II} that there are
isomorphisms $\mu_k :  C^*_r\left(R^+_{\phi}(k)\right) \to \mathbb
D_k$ and $\mu_{k+1} :  C^*_r\left(R^+_{\phi}(k+1)\right) \to \mathbb
D_{k+1}$. There is therefore a unique $*$-homomorphism $\Phi : \mathbb
D_k \to \mathbb D_{k+1}$ such that
\begin{equation}\label{diag32II}
\begin{xymatrix}{
C^*_r\left(R^+_{\phi}(k)\right) \ar[d]_-{\rho} \ar[r]^-{\mu_k} &
\mathbb D_k  \ar[d]^{\Phi}\\
C^*_r\left(R^+_{\phi}(k+1)\right) \ar[r]^-{\mu_{k+1}} &
\mathbb D_{k+1} 
}\end{xymatrix}
\end{equation}
commutes. The $*$-homomorphism $\Phi$ is given by the formula
\begin{equation}\label{form1II}
\Phi(a,f) = (\chi(a) +\mu(f), \varphi(f))
\end{equation}
where 
$$
\chi : \mathbb A_k  \to \mathbb A_{k+1}, \ \ \mu : C\left([0,1],\mathbb B_k\right) \to \mathbb A_{k+1}, \ \ \varphi : C\left([0,1],\mathbb B_k\right) \to C\left([0,1],\mathbb B_{k+1}\right)
$$
are $*$-homomorphisms such that $\chi$ and $\mu$ have orthogonal
ranges. We describe them one by one. The easiest to define is $\chi$;
it is simply given by the formula
$$
\chi(e_{x,y}) = e_{x,y}
$$
when $\phi^k(x) = \phi^k(y) \in \mathcal D$ and
$\val\left(\phi^k,x\right) = \val\left(\phi^k,y\right)$. 

To define $\mu$ and $\varphi$ we use the homeomorphisms $\psi_J :
]0,1[ \to J \in \mathcal I$ that were introduced in Section
\ref{extMar3II}. Consider $f \in C[0,1]$ and a pair $(I,J) \in \mathcal
I_k^{(2)}$. Then
$$
\mu(f \otimes e_{I,J}) = \sum_{(x,y) \in N_{I,J}}  \ f\left( {\psi_{\phi^k(I)}}^{-1}\left(\phi^k(x)\right)\right) e_{x,y}
$$
where 
$$
N_{I,J} = \left\{(x,y) \in \phi^{-k-1}(\mathcal D)^2 : \ x \in
  I, \ y \in J, \ \phi^k(x) = \phi^k(y) \notin \mathcal D\right\} .
$$ 
Similarly, $\varphi$ is given by
\begin{equation}\label{form2II}
\varphi\left( f \otimes e_{I,J}\right) = \sum_{(I_1,J_1) \in M_{I,J}}
f_{I_1}  \otimes
      e_{I_1,J_1}
\end{equation}
where
$$
M_{I,J} = \left\{ (I_1,J_1) \in \mathcal I_{k+1}^{(2)} : \ I_1
  \subseteq I, \ J_1 \subseteq J \right\} 
$$
and $f_{I_1} \in C[0,1]$ is the continuous extension of the function
$$
]0,1[ \ni t \mapsto f\circ \psi_{\phi^k(I)}^{-1} \circ \left(
  \phi|_{\phi^k(I) \cap \phi^{-1}\left(\phi^{k+1}(I_1)\right)}\right)^{-1} \circ
  \psi_{\phi^{k+1}(I_1)}(t).
$$
(We remark that this description of $\rho$ is almost identical with the
one given in Lemma 3.5 in \cite{Th5} in a different setting.)


Define $p_j : \mathbb D_j \to \mathbb A_j$ such that $p_j(a,f) = a$, and note that $p_{k+1} \circ \Phi$ is homotopic to $\left(\chi
  + \mu \circ I_k\right) \circ p_k$ so that
\begin{equation}\label{diag43II}
\begin{xymatrix}{
 \mathbb D_k
\ar[r]^-{p_k} \ar[d]^-{\Phi}
& \mathbb A_k \ar[d]^-{\chi + \mu \circ I_k}  \\
 \mathbb D_{k+1} \ar[r]^-{p_{k+1}}
& \mathbb A_{k+1}   
}\end{xymatrix}
\end{equation}
commutes up to homotopy. It follows that 
$$
\rho_0 = \chi_0  + \mu_0
\circ \left(I_k\right)_0
$$ on $K_0(\mathbb D_k) = \ker \left(\left(I_k\right)_0 - \left(U_k\right)_0\right)$. 

\subsection{An algorithm for the calculation of $K_*\left(C^*_r\left(\Gamma^+_{\phi}\right)\right)$}
For each $(d,v) \in \mathcal D(\pm)$ let $[d,v]$ be the corresponding
element in the standard basis for $\mathbb Z^{\mathcal D(\pm)}$. For
each $d \in \mathcal D$ we let $I^+_d$ be the interval $I^+_d \in
\mathcal I$ such that $d < I^+_d$. We can
then define a homomorphism $A : \mathbb Z^{\mathcal D(\pm)} \to \mathbb
Z^{\mathcal D(\pm)}$ such that
$$
A[d,v] = \left[\phi(d), \val\left(\phi,d\right)\right]
$$
when $v = (+,-)$,
\begin{equation*}\label{formA}
\begin{split}
&A[d,v] =  \left[\phi(d),
    \val\left(\phi,d\right)\right]  \\
& \ \ \ \ \ \ \ \ \ \ \ \ + \sum_{z \in I^+_d \cap
    \phi^{-1}(\mathcal D)} \ \left( \left[\phi(z), \val\left(
        \phi,z\right) \bullet (+,+)\right] +  \left[\phi(z), \val
      \left( \phi, z\right) \bullet (-,-)\right]\right)
\end{split}
\end{equation*}
when $v = (-,+)$,
\begin{equation*}\label{formB}
A[d,v] =  \left[\phi(d),
    \val\left(\phi,d\right)\right] \ \ \  + \sum_{z \in I^+_d \cap
    \phi^{-1}(\mathcal D)} \  \left[\phi(z), \val\left(
        \phi,z\right) \bullet (+,+)\right] 
\end{equation*}
when $v = (+,+)$ and finally
\begin{equation*}\label{formA}
A[d,v] =  \left[\phi(d),
    \val\left(\phi,d\right)\right] \ \ \ + \sum_{z \in I^+_d \cap
    \phi^{-1}(\mathcal D)} \  \left[\phi(z), \val\left(
        \phi,z\right) \bullet (-,-)\right] 
\end{equation*}
when $v = (-,-)$.
 
 Under our identification of $K_0\left(\mathbb
  A_k\right)$ and $K_0(\mathbb A_{k+1})$ with $\mathbb Z^{\mathcal
  D(\pm)}$, the homomorphism $\rho_0 = \chi_0  + \mu_0
\circ  \left(I_k\right)_0$ is given by $A$. In the following we let $\tilde{A}$
denote the restriction of $A$ to $\ker \left( \left(I_k\right)_0 - \left(U_k\right)_0\right)
  \subseteq \mathbb Z^{\mathcal
  D(\pm)}$.

Since $\left(i_j\right)_1 : K_1\left(S\mathbb B_j\right) \to
K_1\left(\mathbb D_j\right)$ is surjective and $K_1\left(S\mathbb
  B_k\right)$ is free there is automatically a homomorphism $B :
K_1\left(S\mathbb B_k\right) \to K_1\left(S\mathbb B_{k+1}\right)$
such that the diagram
\begin{equation}\label{diag43bII}
\begin{xymatrix}{
K_1\left(S\mathbb B_k\right) \ar[d]_-{B} \ar[r]^-{{(i_k})_1}  &
K_1(\mathbb D_k) \ar[d]^-{\Phi} \\
K_1\left(S\mathbb B_{k+1}\right)  \ar[r]^-{{(i_{k+1})}_1}  & K_1\left(\mathbb D_{k+1}\right)  
}\end{xymatrix}
\end{equation}
commutes. In fact, there are generally many choices since
$\left(i_{k+1}\right)_1$ has a kernel. We describe now a way to choose
$B$ such that it is easy to determine from $\phi$. To this end we need some notation. For each $I \in \mathcal I$, $v \in
\left\{(+,+),(-,-)\right\}$, let $[I,v]$ be the corresponding standard
basis element in $\mathbb Z^{\mathcal I(\pm)}$. For $v \in \left\{(+,+), (-,-)\right\}$, set
$$
(-1)^v = \begin{cases} 1 & \ \text{when} \ v = (+,+) \\ -1 & \
  \text{when} \ v = (-,-) . \end{cases}
$$

Let $J \in \mathcal I$ and choose $J_1 \in \mathcal I$ such that $J_1
\subseteq \phi(J)$. We choose $J_1$ such that $J_1 = \phi(J_X)$
where $J_X \subseteq J$ is an open subinterval and $J_X \cap
\mathcal C_1 = \emptyset$. Let $I \in \mathcal I_k$ be such that $\phi^k(I) =
J$ and $\val\left( \phi^k,I\right) =  v \in (\pm,\pm)$. Since $\phi^k$ is
injective on $I$ there is a unique open subinterval $I_1 \subseteq I$
such that $\phi^k(I_1) = J_X$. Note that $\phi^{k+1}$ is injective on $I_1$
and $\phi^{k+1}(I_1) = J_1$. In particular, $I_1 \in \mathcal
I_{k+1}$. Let $u \in C_c(I_1) \subseteq C_c\left(R^+_{\phi}(k)\right)$
be a positively oriented path in $\mathbb T - 1$ of degree 1. Then 
$$
(u \circ \lambda_I \circ \psi_J) \otimes e_{I,I}
$$
represents $(-1)^v [J,v] \in \mathbb Z^{\mathcal I(\pm)} \simeq K_1(S\mathbb B_k)$. By using (\ref{form1II}) and
(\ref{form2II}) we find that $\mu_{k+1} \circ \rho(u) \in S\mathbb
B_{k+1} \subseteq \mathbb D_{k+1}$, and in $S\mathbb B_{k+1}$ is the element 
$$
\left(u \circ \lambda_I \circ \left(\phi|_{\phi^k({I}) \cap
    \phi^{-1}
    \left(\phi^{k+1}\left({I_1}\right)\right)}\right)^{-1}
\circ \psi_{J_1}\right) \otimes e_{I_1,I_1} = \left(u \circ \lambda_{I_1} \circ
    \psi_{J_1}\right) \otimes e_{I_1,I_1}
$$
which corresponds to $(-1)^v (-1)^w \left[J_1,w \bullet v\right]$ in
$\mathbb Z^{\mathcal I(\pm)}$ where $w =
\val \left(\phi,J_X\right)$. Therefore a recipe for a construction of
$B$ reads as follows: 
For each
$J \in \mathcal I$ choose an open interval $J' \subseteq J$ such that
$J' \cap \mathcal C_1 = \emptyset$ and $\phi(J') \in \mathcal I$. Then
$$
B[J,v] = (-1)^{\val(\phi,J')} \left[\phi(J'), \val\left(\phi,J'\right)\bullet v\right] .
$$
The construction of $B$ generally depends on many choices which lead
to different homomorphisms. But the commutativity of (\ref{diag43bII})
guarantees that they all leave $\im \left(\left(I_k\right)_0 - \left(U_k\right)_0\right)
\subseteq \mathbb Z^{\mathcal I(\pm)}$
globally invariant and induce the same endomorphism $\tilde{B}$ of $\coker
\left(\left(I_k\right)_0 - \left(U_k\right)_0\right)$.

In view of the 6 terms exact sequence (\ref{6termIIk})
it now follows that the endomorphisms 
\begin{equation*}
\begin{split}
&\tilde{A} : \ker
\left(\left(I_k\right)_0 - \left(U_k\right)_0\right) \to \ker
\left(\left(I_k\right)_0 - \left(U_k\right)_0\right) , \\
&  \tilde{B} : \coker
\left(\left(I_k\right)_0 - \left(U_k\right)_0\right) \to \coker
\left(\left(I_k\right)_0 - \left(U_k\right)_0\right) 
\end{split}
\end{equation*}
determine
$K_0\left(C^*_r\left(\Gamma^+_{\phi}\right)\right)$ and
$K_1\left(C^*_r\left(\Gamma^+_{\phi}\right)\right)$ in the sense that
there are extensions
\begin{equation}\label{diag34II}
\begin{xymatrix}{
0 \ar[r] & \coker \left( 1- \tilde{A}\right) \ar[r]
& K_0\left(C^*_r\left(\Gamma^+_{\phi}\right)\right) \ar[r] & \ker
\left(1 - \tilde{B}\right) \ar[r] &  0
}\end{xymatrix}
\end{equation}
and
\begin{equation}\label{diag35II}
\begin{xymatrix}{
0 \ar[r] & \coker \left( 1- \tilde{B}\right) \ar[r]
& K_1\left(C^*_r\left(\Gamma^+_{\phi}\right)\right) \ar[r] & \ker
\left(1 - \tilde{A}\right) \ar[r] &  0 .
}\end{xymatrix}
\end{equation}
Note that the last extension is always split and hence
$$
K_1\left(C^*_r\left(\Gamma^+_{\phi}\right)\right) \simeq \coker \left(
  1- \tilde{B}\right) \oplus \ker
\left(1 - \tilde{A}\right) .
$$

To identify the $C^*$-algebra from its $K$-theory groups it is
important to know which element of
$K_0\left(C_r^*\left(\Gamma^+_{\phi}\right)\right)$ represents the
unit $1$ of $C^*_r\left(\Gamma^+_{\phi}\right)$. Note therefore that
$[1] \in K_0\left(C^*_r\left(\Gamma^+_{\phi}\right)\right)$ is the
image of $[1] \in K_0\left(C^*_r\left(R^+_{\phi}(k)\right)\right)$
under the map $\iota_0$ in (\ref{6termIIk}). Under the identification
$K_0\left(\mathbb A_k\right) = \mathbb Z^{\mathcal
  D(\pm)}$ we have that
$$
[1] = \sum_{(d,v) \in \mathcal D(\pm)} m(d,v) [d,v]
$$
in $K_0\left(\mathbb A_k\right)$, where $m(d,v) = \#\left\{ x \in \phi^{-k}(d): \
  \val\left(\phi^k,x\right) = v \right\}$. Since $I_k$ and $U_k$ are
unital $*$-homomorphisms this element is always in
$\ker\left(\left(I_k\right)_0 - \left(U_k\right)_0\right)$ and gives
therefore rise to an element of $\coker \left(1- \tilde{A}\right)$
which under the embedding $\coker \left(1- \tilde{A}\right) \subseteq
K_0\left(C_r^*\left(\Gamma^+_{\phi}\right)\right)$ from
(\ref{diag34II}) gives us the
element representing $[1] \in K_0\left(C^*_r\left(\Gamma^+_{\phi}\right)\right)$.

\begin{example}\label{exx1} In this example we show with a fair amount
  of details how to complete the K-theory calculations for a family of
  Markov maps by using the recipe described above. Let $m,k \in \mathbb N$ and let $g_{m,k} : [0,1] \to
  \mathbb R$ be the continuous piecewise linear map with the
  properties that $g_{m,k}(0) = 0$, $g_{m,k}$ has slope $2m$ on
  $\left[0,\frac{1}{2}\right]$ and slope $-2k$ on
  $\left[\frac{1}{2},1\right]$. Then $g_{m,k}$ is the lift of a
  piecewise monotone map $\phi_{m,k}$ on the circle which is exact and not locally injective. To simplify the calculations we assume that $m \geq 2, k \geq 2$ or that $m =k$. Then $\phi_{m,k}$ is Markov of order $1$ and $C^*_r\left(\Gamma^+_{\phi_{m,k}}\right)$ is simple, unital, separable,
  purely infinite and satisfies the UCT. By the Kirchberg-Phillips
  classification result the algebra is therefore determined by its K-theory
  groups and the position of the unit in the $K_0$-group.

Note that $1$ is the only
critical value and that
$$
\mathcal D(\pm) = 1 \times \mathcal V = \{ (1,(-,+)), (1,(+,-)),
(1,(+,+)), (1,(-,-))\}.
$$
There is only one interval in $\mathcal I$, namely $I = \mathbb T
\backslash \{1\}$, and $\mathcal I(\pm) = \left\{(I,(+,+)), (I,(-,-)\right\}$. When we take the elements of $\mathcal D(\pm)$ and $\mathcal I(\pm)$
in that order we find that
$$
\left(I_1\right)_0 = \left( \begin{matrix} 1 & 0  & 1 & 0 \\ 1 & 0 & 0
    & 1    \end{matrix} \right), \\ 
\left(U_1\right)_0 = \left( \begin{matrix} 0 & 1 & 1 & 0  \\ 0 & 1 & 0 & 1  \end{matrix} \right),
$$
It follows that 
$$
\ker \left(\left(I_1\right)_0 - \left(U_1\right)_0\right) = \left\{ (z,z,u,v) \in \mathbb Z^4 :
  \ z,u,v \in \mathbb Z \right\} \simeq \mathbb Z^3,
$$ 
and 
$$
\coker \left(\left(I_1\right)_0 - \left(U_1\right)_0\right) \ \simeq
\ \mathbb Z.
$$
The matrix $A$ is 
$$
A = \left( \begin{matrix} 1 & 1 & 1 & 1 \\ 2 & 0 & 1 & 1 \\ m+k-2 & 0 &
    m-1& k-1 \\ m+k-2 & 0 & k-1 & m-1 \end{matrix} \right).
$$
and its restriction $\tilde{A}$ to $\ker \left( \left(I_1\right)_0 -
  \left(U_1\right)_0\right) = \mathbb Z^3$ is
$$
\tilde{A} = \left( \begin{matrix} 2 & 1 & 1  \\ m+k-2 &
    m-1& k-1 \\ m+k-2 & k-1 & m-1 \end{matrix} \right).
$$
Hence 
$$
\tilde{A} -1 = \left( \begin{matrix} 1 & 1 & 1  \\ m+k-2 &
    m-2& k-1 \\ m+k-2 & k-1 & m-2 \end{matrix} \right).
$$
After a couple of row-operations we get from $\tilde{A} -1$ the matrix
$$
\left( \begin{matrix} 1 & 1 & 1  \\ 0 &
   k & m-1 \\ 0 & m-1 & k \end{matrix} \right).
$$
Hence $\ker \left(\tilde{A}-1\right) = 0$ when $k \neq m-1$ while
$\ker \left(\tilde{A} -1 \right) \simeq \mathbb Z$ when $k = m-1$.
To determine $\coker \left( \tilde{A} -1 \right)$ note that this
cokernel is the same as the cokernel of 
$$
\left(\begin{matrix} k & m-1\\ m-1 & k\end{matrix} \right) .
$$
Let $g$ be the greatest common divisor of $k$ and $m-1$ when $m \geq 1$ and set $g=k$ when $m=1$. There are then
$x,y \in \mathbb Z$ such that $x \frac{k}{g} + y\frac{m-1}{g}  =
1$. Then
$$
\left(\begin{matrix} x & y\\ -\frac{m-1}{g}  & \frac{k}{g}\end{matrix}
\right) \in GL_2(\mathbb Z)
$$
and
$$
\left(\begin{matrix} x & y\\ -\frac{m-1}{g}  & \frac{k}{g}\end{matrix} \right) \left(\begin{matrix} k & m-1\\ m-1 & k\end{matrix} \right) = 
\left(\begin{matrix} g & 0\\ 0 & g\end{matrix} \right)
\left(\begin{matrix} 1 &  x\frac{m-1}{g} + y \frac{k}{g} \\  0 &
    \frac{k^2 - (m-1)^2}{g^2} \end{matrix} \right)
$$
After a final column operation this shows that the Smith normal form of $
\left(\begin{smallmatrix} k & m-1\\ m-1 & k\end{smallmatrix} \right)$
is
$$
\left(\begin{matrix} g & 0 \\ 0 & \frac{k^2 - (m-1)^2}{g} \end{matrix} \right) .
$$
Hence
$$
\coker \left( \tilde{A} -1 \right) \simeq \mathbb Z_g \oplus \mathbb
Z_{\frac{k^2 - (m-1)^2}{g}} .
$$
From the recipe for the matrix $B$ it is easily seen that we can
choose it to be the identity matrix in this case. It follows that the endomorphism
$\tilde{B}$ of $\coker \left(\left(I_1\right)_0 -
  \left(U_1\right)_0\right)$ is the identity and hence $\ker
\left(\tilde{B} - 1\right) \simeq \coker\left(\tilde{B} -1 \right)
\simeq \coker \left(\left(I_1\right)_0 -
  \left(U_1\right)_0\right) \simeq \mathbb Z$. 
All in all we get the conclusion that
$$
K_0\left(C^*_r\left(\Gamma^+_{\phi_{m,k}}\right)\right) \simeq \mathbb
Z^2 \oplus \mathbb Z_{m-1}, \ \
K_1\left(C^*_r\left(\Gamma^+_{\phi_{m,k}}\right)\right) \simeq \mathbb
Z^2
$$
when $k = m-1$ and
$$
K_0\left(C^*_r\left(\Gamma^+_{\phi_{m,k}}\right)\right) \simeq \mathbb
Z \oplus \mathbb Z_g \oplus \mathbb Z_{\frac{k^2 - (m-1)^2}{g}}, \ \
K_1\left(C^*_r\left(\Gamma^+_{\phi_{m,k}}\right)\right) \simeq \mathbb
Z
$$
when $k \neq m-1$. In the first case the class of the unit is
represented by $m-1 \in \mathbb Z_{k^2 - (m-1)^2}$ and in the second
by $\left(-x,\frac{m-1}{g}\right) \in \mathbb Z_g \oplus \mathbb
Z_{\frac{k^2 - (m-1)^2}{g}}$.

Let $\phi_{-m,-k}(z) = \overline{\phi_{m,k}(z)}$. Then $\phi_{-m,-k}$
is also exact and Markov, and the calculation of
$K_*\left(C^*_r\left(\Gamma^+_{\phi_{-m,-k}}\right)\right)$ can
proceed exactly as above. In comparison the roles of $k$ and $m$ are
interchanged and we find that
$$
K_0\left(C^*_r\left(\Gamma^+_{\phi_{-m,-k}}\right)\right) \simeq \mathbb
Z^2 \oplus \mathbb Z_{k-1}, \ \
K_1\left(C^*_r\left(\Gamma^+_{\phi_{-m,-k}}\right)\right) \simeq \mathbb
Z^2
$$
when $m = k-1$ and
$$
K_0\left(C^*_r\left(\Gamma^+_{\phi_{-m,-k}}\right)\right) \simeq \mathbb
Z \oplus \mathbb Z_{g'} \oplus \mathbb Z_{\frac{m^2 - (k-1)^2}{g'}}, \ \
K_1\left(C^*_r\left(\Gamma^+_{\phi_{-m,-k}}\right)\right) \simeq \mathbb
Z
$$
when $m \neq k-1$, where $g'$ is now the greatest common divisor of
$m$ and $k-1$, or $g'= m$ when $k=1$.

\end{example}

\end{document}